\documentclass[11pt,reqno,twoside]{amsart}

\usepackage{amsfonts,amsmath,amssymb}
\usepackage{mathrsfs,mathtools}
\usepackage{enumerate}
\usepackage[colorlinks=true, citecolor=black, urlcolor=black, linkcolor=black, pdfborder={0,0,0}]{hyperref}
\usepackage{cleveref}
\usepackage{xcolor}
\usepackage{esint}
\usepackage{graphicx}
\usepackage{bm,enumitem}
\usepackage{commath}
\usepackage{esint}
\usepackage{cases}
\DeclareMathAlphabet{\mathpzc}{OT1}{pzc}{m}{it}
\usepackage[colorlinks=true]{hyperref}
\usepackage[a-1b]{pdfx}
\usepackage{mathtools}
\usepackage[scr=boondoxupr]{mathalpha}
\newcommand*{\Lcorner}{%
    \mathchoice%
        {\mathrel{\makebox[7pt][c]{\rule{.4pt}{7.5pt}\rule{5pt}{.4pt}}}}%
        {\mathrel{\makebox[7pt][c]{\rule{.4pt}{7.5pt}\rule{5pt}{.4pt}}}}%
        {\mathrel{\makebox[5.5pt][c]{\rule{.4pt}{5.25pt}\rule{3.5pt}{.4pt}}}}%
        {\mathrel{\makebox[4pt][c]{\rule{.4pt}{3.75pt}\rule{2.5pt}{.4pt}}}}%
}

\usepackage{placeins} 
\usepackage{flafter} 
\usepackage{mathrsfs}

\usepackage[textsize=small]{todonotes}
\definecolor{darkgreen}{RGB}{0,95,0}

\usepackage{lineno}

\numberwithin{equation}{section}

\usepackage[dvips,bottom=1in,right=1in,top=1in, left=1in]{geometry}

\newtheorem{theorem}{Theorem}[section]
\newtheorem{corollary}[theorem]{Corollary}
\newtheorem{definition}[theorem]{Definition}
\newtheorem{lemma}[theorem]{Lemma}

\newtheorem{proposition}[theorem]{Proposition}

\newtheorem{remark}[theorem]{Remark}
\newtheorem{example}{Example}
\crefformat{equation}{(#2#1#3)}

\renewcommand{\div}{\mathrm{div}}

\usepackage{graphicx}

\thanks{
This work is  supported by NSF grant DMS-2012391}

\begin{document}

	\title[Sets of measures with divergences and boundary conditions]{{Convergence aspects for sets of measures with divergences and boundary conditions}
	}
	
	\author{Nicholas Chisholm, Carlos N. Rautenberg}
	\address{N. Chisholm, C.N. Rautenberg. Department of Mathematical Sciences and the Center for Mathematics and Artificial Intelligence (CMAI), George Mason University, Fairfax, VA 22030, USA.}
	\email{nchishol@gmu.edu, crautenb@gmu.edu}
	
	\begin{abstract}
		In this paper we study set convergence aspects for Banach spaces of vector-valued measures with divergences (represented by measures or by functions) and applications. We consider a form of normal trace characterization to establish subspaces of measures that directionally vanish in parts of the boundary, and present examples constructed with binary trees. Subsequently we study convex sets with total variation bounds and their convergence properties together with applications to the stability of optimization problems.
	\end{abstract}
	
		
	\maketitle

	\tableofcontents
	
	\section{Introduction}
	
	The purpose of the paper is severalfold and closely tied with applications. { In particular, two major aspects are considered; initially we focus on} (i) The  description and study of subspaces of the space of Borel measures over a subset $\Omega\subset \mathbb{R}^M$ with (measure and functional) divergences that can be characterized as directionally vanishing in parts of the boundary $\partial\Omega$. { Secondly, we approach}  (ii) The study of set convergence aspects of sets of measures whose total variations are bounded by non-negative measures and their application to stability of optimization problems.
	
	The need to represent directional boundary conditions on certain classes of Borel measures arises in the Fenchel dualization of non-dissipative gradient constraints problems; see \cite{antil2022nondiffusive}. The latter class of problems allows to model the growth of sandpiles and granular material flow in a deterministic fashion. In this setting, the region where measures should vanish directionally at the boundary corresponds to the region where material is not allowed to escape the domain.
	
	Optimization problems over spaces of measures and with total variation constraints are relatively scarce in the literature. Notable exceptions can be found in \cite{buttazzo2013shape}, where applications to shape optimization with  { total-variation-norm constraints are}  considered, and in \cite{buttazzo2022mass}, where minimizers to constant total  { variation norm} problems with convex energies are characterized by means of suitable PDEs.
	
	Concerning vector fields with generalized divergences, the seminal work and several extensions were developed by Chen and Frid \cite{chen1999divergence,chen2001theory,chen2003extended}. The authors establish properties of the functional spaces, the Gauss-Green theorem in this setting, and the study of trace type results. Their original motivation is the study of hyperbolic conservation laws. Generalizations of the trace results and the integration by parts theorems were established by  \v{S}ilhav\'{y} \cite{vsilhavy2008divergence} where best possible cases are determined for the normal trace results.
	
	Although the study of set convergence goes back to Painlev\'{e}, see the Painlev\'{e}-Kuratowski set limits in \cite{kuratowski1948topologie}, the appropriate concept for the study of perturbations of constrained optimization problems and variational inequalities in reflexive Banach spaces  was developed by Mosco \cite{mosco1969convergence,mosco1967approximation}. The main object of study of our work is the following set $\mathbf{K}(\alpha;X)$ defined as
	\begin{equation*}
			\mathbf{K}(\alpha; X) := \{ \mu \in X : { |\mu|} \le \alpha \}, 
	\end{equation*}
	 { where the expression $|\mu|\leq \alpha$ stands for the total variation of $\mu$ (see \Cref{sec:prelim} for details) dominated by a non-negative Borel measure  $\alpha$}, and $X$ is a subspace of the Banach space of Borel measures endowed with the total  { variation} norm. In particular, we focus on properties of the map $\alpha\mapsto \mathbf{K}(\alpha; X)$.

	The paper is organized as follows. Initially, we present a formal motivation for the class of spaces that we will study in \Cref{sec:formal}.  In \Cref{sec:prelim} we provide the notation and conventions used throughout the entire paper, in particular we consider the three different topologies on the space of Borel measures that we require in our {  approach}, the strong, narrow, and weak topologies. In \Cref{sec:Mdiv}, we establish the spaces of { vector} measures with divergences that are either represented by measures or functions in some Lebesgue spaces, and present a known trace characterization. Subsequently, in \Cref{sec:BC} and \Cref{sec:BC2} we introduce the subspaces with {  generalized} normal-traces { (understood as the generalization of evaluations at the boundary of  normal components to the boundary)}  vanishing on subsets of the boundary, and the construction of measures by means of binary trees. Order properties and equivalent characterizations thereof needed for the definition of the convex sets of interest are given in \Cref{sec:sets}, and the set convergence results are given in \Cref{sec:M1} and \Cref{sec:M2}.

	\subsection{Formal motivation} \label{sec:formal}
	The Prigozhin mathematical model \cite{MR3082292,MR3231973,Prigozhin1996,PrigozhinSandpile} of cohensionless and granular material growth over a certain flat surface given by $\Omega$ can be formulated as an {  evolving in time} gradient constrained problem without dissipative operators. For boundary conditions, one considers a region of the boundary $\Gamma$ where material is not allowed to leave the domain and $\partial\Omega\setminus\Gamma$  where material is allowed to leave freely. The semi-discretization of the model and the formal determination of the Fenchel dual problem in each time step leads to try to identify a Borel measure $\mu$ in the following  class of minimization problems:
	\begin{equation*} 
			\begin{aligned}
			&\min_{\mu} \quad  \frac{1}{2} \int_{\Omega} \left| \mathrm{div}\, \mu(x) - f(x)\right|^2 \, \mathrm{d} x   
			+ \int_{\Omega} \beta(x) \, \mathrm{d}|\mu|(x)\\
			&\text{ over  the set of Borel measures on }\Omega\\
			&\textrm{ subject to (s.t.)} \quad  \mu\cdot \vec{n}=0 \text{ on } \Gamma  \text{ (in some sense)} \text{ and } |\mu| \le \alpha,
			\end{aligned}
		\end{equation*}
where $f\in L^2(\Omega)$,  $\beta \in C(\overline{\Omega})^+$; see \cite{antil2022nondiffusive}. {  As stated before,} the expression $|\mu|\leq \alpha$ stands for the total variation of $\mu$ dominated by some measure $\alpha$ that is  non-negative. The measure $\alpha$ may arise as a structural constraint, that is, it may be related to a finite element mesh, and hence determined by a linear combination of Dirac deltas (element nodes), Lebesgue one{ -}dimensional measures (element edges), and functions (element areas). While initially the entire formulation of the problem is formal, we show in \Cref{optimize_static} that the problem can be posed rigorously and admits solutions. In addition, \Cref{optimize_dynamic} shows that the problem  is stable with respect to perturbations on $\alpha$ with respect to the total variation norm.

	\section{Notation and Preliminaries} \label{sec:prelim}

	Let $\Omega$ be an open subset of $\mathbb{R}^M$ with $M\in\mathbb{N}$, and $\mathcal{B}(\Omega)$ be the Borel $\sigma$-algebra on $\Omega$. We call elements of $\mathcal{B}(\Omega)$  \textit{Borel sets}. Let $\mathrm{M}(\Omega)$ and $\mathrm{M}(\Omega)^N$ denote the set of all real-valued and $\mathbb{R}^N$-valued measures with $N\in\mathbb{N}$, respectively,  on $\mathcal{B}(\Omega)$. An element $\sigma \in \mathcal{B}(\Omega)$ is \textit{positive} 
	if for every $B \in \mathcal{B}(\Omega)$, we have $\sigma(B)\geq 0$. We 
	use $\mathrm{M}^{+}(\Omega)$ to denote the set of all positive measures 
	on $\mathcal{B}(\Omega)$. The \textit{total variation}
	of a measure $\mu \in \mathrm{M}(\Omega)^N$ is 
	the uniquely defined measure $|\mu| \in \mathrm{M}^{+}(\Omega)$ that satisfies 
	\begin{equation*}
		|\mu|( B ) = 
		\sup \left\{ \sum_{i = 1}^{\infty} |\mu(B_i)|  : 
		B = \bigcup_{i = 1 }^{\infty } B_i \right\} \quad \text{for all} \quad B \in \mathcal{B}(\Omega),
	\end{equation*}
	see \cite{abm14}. Recall that $\mathrm{M}(\Omega )^N$ is a Banach space when endowed with the norm 
	\begin{equation}
		\| \mu \|_{\mathrm{M}(\Omega )^N} := |\mu|(\Omega).
	\end{equation}
	Further, by duality of the set of continuous functions with compact support $C_c( \Omega)$ and $\mathrm{M}(\Omega )$ {(  see Section 2.4 in \cite{abm14})} we observe that 	
	\begin{equation}\label{eq:TVmu}
		|\mu|(\Omega) = \sup 
		\{ \langle \mu , \phi \rangle 
		: \phi \in C_c( \Omega)^N  : | \phi(x) | \le 1 \text { for all } x\in \Omega\}, 
	\end{equation} 
	where the pairing in \eqref{eq:TVmu} between $\mu$ and $\phi$ is given by
	\begin{equation*}
		\langle \mu , \phi \rangle=\int_{\Omega} \phi\cdot  \dif \mu:=\sum_{i=1}^N\int_{\Omega} \phi_i \dif \mu_i,
	\end{equation*}
	with $\phi=\{\phi_i\}_{i=1}^N$ and $\mu=\{\mu_i\}_{i=1}^N$. Note that in the definition of \eqref{eq:TVmu}, one could substitute the space $C_c(\Omega)$ for $C_0(\Omega)$ without changing the value of the supremum where $C_0(\Omega)$ is the space of continuous functions vanishing at the boundary $\partial\Omega$ and equipped with the usual $\phi\mapsto\sup_{x\in\Omega} |\phi(x)|$ norm.

	If  a sequence $\{\mu_n\}$ in $\mathrm{M}(\Omega)^N$ converges to $\mu\in \mathrm{M}(\Omega)^N$ in norm, that is
	\begin{equation*}
		\| \mu_n - \mu \|_{\mathrm{M}(\Omega )^N} \to 0
	\end{equation*}
	as $n\to\infty$, we say that $\{\mu_n\}$ converges \emph{strongly} to $\mu$ and we denote this by 
	\begin{equation*}
		\mu_n \to \mu.
	\end{equation*}
	In addition to the topology induced by the norm $\mu\mapsto |\mu|(\Omega)$, two other topologies on $\mathrm{M}(\Omega )^N$  are of interest: the \textit{narrow} and the \textit{weak} topologies.

	\begin{definition}[\textsc{Narrow, and Weak Convergence in $\mathrm{M}(\Omega )^N$}]\label{def:NW}
	Let $\{\mu_n\}$ be a sequence of measures in 
	$\mathrm{M}(\Omega)^N$ with $\mu\in \mathrm{M}(\Omega)^N$. 
	If for all  $\phi
	\in C_b( \Omega )^N$, where $C_b( \Omega )$ is the set of bounded continuous functions on $\Omega$, we observe 
	\begin{equation*}
		\int_{\Omega} \phi \cdot \mathrm{d} \mu_n 
		\to \int_{\Omega} \phi \cdot \mathrm{d} \mu
	\end{equation*}
	as $n\to\infty$, we say that $\{\mu_n\}$ converges \emph{narrowly} to $\mu$ and write
	\begin{equation*}
		\mu_n\xrightarrow[]{\mathrm{nw}}\mu.
	\end{equation*}
Further, we say that $\{\mu_n\}$ converges \emph{weakly} to $\mu$ if 
	\begin{equation*}
		\int_{\Omega} \phi \cdot \mathrm{d} \mu_n 
		\to \int_{\Omega} \phi \cdot \mathrm{d} \mu
	\end{equation*}
	for all $ \phi \in C_c( \Omega)^N$ as $n\to\infty$, in which case we write 
	\begin{equation*}
		\mu_n\rightharpoonup \mu.
	\end{equation*} 
	\end{definition}
	
Our terminology is the one used in \cite{abm14}.
To avoid confusion, note that our definition of \textit{narrow convergence} is called \emph{weak convergence}  by Bogachev \cite{bogachev_vol2}, and other authors.

\section{Spaces of { Vector} Measures with Divergences} \label{sec:Mdiv}

In this section, we consider a subset of vector valued measures in $\mathrm{M}(\Omega)^N$ with $\Omega\subset \mathbb{R}^N$ that admit a weak divergence that is defined as a 
measure in $\mathrm{M}(\Omega)$ or that can be identified as a function in $L^p(\Omega)$. 
The reader is referred to work of Chen and Frid \cite{chen1999divergence,chen2003extended,chen2001theory} and \v{S}ilhav\'{y} \cite{vsilhavy2008divergence} for the seminal work on vector fields with generalized divergences and extensions. These subspaces are of particular interest, as they allow for a definition of a normal trace integral at the boundary $\partial\Omega$ of $\Omega$; { provided $\partial\Omega$ exists, note that we have only assumed that $\Omega$ is an open set}. Furthermore, the latter is fundamental in defining measures with zero normal traces. Consider the following initial definition.
\begin{definition}\label{def:DM}
	We define  $\mathrm{DM}(\Omega)$ as the set of all  $\mu \in \mathrm{M}(\Omega)^N$ for which there exists a measure $\sigma\in \mathrm{M}(\Omega)$ such that
		\begin{equation} \label{eq:divm}
		\int_{\Omega} \nabla \phi \cdot \mathrm{d} \mu 
		= - \int_{\Omega} \phi \:\mathrm{d} \sigma
		\qquad \forall \phi \in C_c^{\infty}(\Omega).
		\end{equation}
We define $\sigma$ to be the \emph{divergence of} $\mu$ and denote
\begin{equation*}
	\mathrm{div}\, \mu :=\sigma.
\end{equation*}
\end{definition}
The subset $\mathrm{DM}(\Omega)$ of $\mathrm{M}(\Omega)^N$ is then a linear space and can be  defined as
\begin{equation*}
		\mathrm{DM}(\Omega):= \{ \mu \in \mathrm{M}(\Omega)^N: \mathrm{div}\, \mu \in \mathrm{M}(\Omega) \}, 
\end{equation*}
which is a Banach space when endowed with the norm 
\begin{equation*}
	\| \mu \|_{\mathrm{DM}(\Omega)} := \| \mu \|_{\mathrm{M}(\Omega )^N} + \|\mathrm{div}\, \mu\|_{\mathrm{M}(\Omega )}.
\end{equation*}
If $\mu\in \mathrm{DM}(\Omega)$ and $\div \mu$ is absolutely continuous with respect to the Lebesgue measure, then we { state} $\mu\in \mathrm{M}^1(\Omega;\div )$. In general, for $1\leq p \leq +\infty$, we define $\mathrm{M}^p(\Omega;\div )$ as the set of all  $\mu \in \mathrm{M}(\Omega)^N$ for which there exists $h\in L^p(\Omega)$ such that 
		\begin{equation} \label{eq:div}
		\int_{\Omega} \nabla \phi \cdot \mathrm{d} \mu 
		= - \int_{\Omega} \phi h \, \mathrm{d} x
		\qquad \forall \phi \in C_c^{\infty}(\Omega),
		\end{equation}
	where ``$\mathrm{d} x$'' denotes integration with respect to the Lebesgue measure and $h:=\mathrm{div}\, \mu $. In this setting, we have
	\begin{equation*}
			\mathrm{M}^p(\Omega;\div ):= \{ \mu \in \mathrm{M}(\Omega)^N: \mathrm{div}\, \mu \in L^p(\Omega) \}, 
		\end{equation*} 
		which is likewise a Banach space when endowed with the norm 
		\begin{equation*}
			\| \mu \|_{\mathrm{M}^p(\Omega;\div )} := \| \mu \|_{\mathrm{M}(\Omega )^N} + \|\mathrm{div}\, \mu\|_{L^p(\Omega)}.
		\end{equation*} 
	Hence we refer to $\mathrm{M}^p(\Omega;\div )$ as the space of { vector} measures with (weak) divergences in $L^p(\Omega)$.  Further, the closed subspace of measures $\mu\in \mathrm{M}^p(\Omega;\div )$ such that $\div \mu$ is identically zero is denoted by $\mathrm{M}(\Omega;\div \,0)$.

	We can relate the previously defined $\mathrm{DM}(\Omega)$ and $\mathrm{M}^p(\Omega;\div )$ with the classical Sobolev space $H^1(\Omega;\div)$ defined as 
	 	\begin{equation*}
			H^1(\Omega;\div):= \{ v \in L^2(\Omega)^N: \mathrm{div}\, v\in L^2(\Omega) \}; 
		\end{equation*} 
		that is, a vector field $v$ in $L^2(\Omega)^N$ belongs to $H^1(\Omega;\div)$ if its weak divergence is in $L^2(\Omega)$. With the usual identifications, we have $H^1(\Omega;\div)\subset\mathrm{DM}(\Omega)$ and $H^1(\Omega;\div)\subset \mathrm{M}^p(\Omega;\div )$ for $1\leq p \leq 2$.

		Due to Chen and Frid \cite{chen1999divergence,chen2003extended,chen2001theory} and \v{S}ilhav\'{y}~\cite{vsilhavy2008divergence}, we observe a form of trace characterization for $\mathrm{DM}(\Omega)$. We denote by $\mathrm{Lip}^{B}(\Lambda)$ the space of Lipschitz maps $z:\Lambda\to\mathbb{R}$ for $\Lambda\subset \mathbb{R}^k$ and endow it with the norm
	\begin{equation*}
		\|z\|_{\mathrm{Lip}^{B}(\Lambda)}:= \mathrm{Lip}(z)+\sup_{x\in \Lambda}|z(x)|,
	\end{equation*}
where $\mathrm{Lip}(z)$ is the Lipschitz constant of $z$ on $\Lambda$. It follows that  for each $\mu\in {\mathrm{DM}(\Omega)}$ there exists a linear functional ${ \mathcal{N}_{\mu}:\mathrm{Lip}^{B}(\mathbb{R}^M)}\left.\right|_{\partial\Omega}\to \mathbb{R}$ such that for all $v\in { \mathrm{Lip}^{B}(\mathbb{R}^M)}$ we have
\begin{equation}\label{eq:IntParts0}
	{ \mathcal{N}}_{\mu}(v|_{\partial \Omega})=\int_{\Omega} { \dif \:\langle\langle \nabla v , \mu\rangle\rangle}+\int_\Omega v\: \dif\:\div\:\mu,
\end{equation}
{  where $\langle\langle \nabla v , \mu\rangle\rangle$ is a scalar measure on $\Omega$ that is absolutely continuous with respect to $\mu$. In the case that $v\in C^1$ is bounded and with bounded derivative,  $\dif\: \langle\langle \nabla v , \mu\rangle\rangle$ in \eqref{eq:IntParts0} can be replaced by $ \nabla v\cdot \dif \mu$. Further, } 
\begin{equation*}
|	{ \mathcal{N}}_{\mu}(g)|\leq \| \mu \|_{\mathrm{DM}(\Omega)} \|g\|_{\mathrm{Lip}^{B}(\partial\Omega)},
\end{equation*}
for all  $g\in \mathrm{Lip}^{B}(\partial\Omega)$. 

{ Based on the previous, we define $\mathrm{N}$ as 
\begin{equation*}
	\mathrm{N}(v,\mu)=\int_{\Omega} \nabla v\cdot \dif \mu+\int_\Omega v\: \dif\:\div\:\mu,
\end{equation*}
for $\mu\in {\mathrm{DM}(\Omega)}$ and $v\in C_b^1(\Omega)$, the space of bounded functions in $C^1(\Omega)$ whose partial derivatives are all bounded as well. A few words are in order concerning $\mathrm{N}(v,\mu)$: Note that since $\mu\in {\mathrm{DM}(\Omega)}$ then $\mu$ and $\mathrm{div}\, \mu $ are Borel measures, and since $v\in C_b^1(\Omega)$ then $v$ and $\nabla v$ are bounded and continuous over $\Omega$, so that $\mathrm{N}(v,\mu)$ is well-defined. In addition, in what follows we also consider 
\begin{equation*}
	\mathbf{C}_b^1(\overline{\Omega}):=C_b^1(\Omega)\cap C(\overline{\Omega}),
\end{equation*}
that is, $\mathbf{C}_b^1(\overline{\Omega})$ is the subspace of $C_b^1(\Omega)$ of functions that can be extended continuously to $\overline{\Omega}$. The latter is used for the definition of a notion of boundary condition for measures in $\mathrm{DM}(\Omega)$. 
}

Provided that $\Omega$ is sufficiently smooth, and $\mu$ and $v$ are sufficiently regular functions, we observe 
\begin{equation*}
	{ \mathrm{N}(v,\mu)}=\int_{{ \partial\Omega}} v\; \mu \cdot \vec{n}  \:\text{d}\,\mathcal{H}^{ {N-1}},
\end{equation*}
where $\vec{n}$ is the outer unit normal vector at $\partial\Omega$. We refer to the map ${ \mathrm{N}(\cdot,\mu)}$ as the \textit{normal trace} of $\mu$ on $\partial\Omega$.  More specifically, if we assume that $\Omega$ has a Lipschitz boundary, then the map $w\mapsto w|_{\partial\Omega}\cdot\vec{n} $ is extended by continuity from $C^\infty(\overline{\Omega})$ to a map from $H^1(\Omega;\div)$ to $H^{-1/2}(\partial\Omega)$. In the latter case, $\mathrm{N}_w(v|_{\partial \Omega})=\langle w\cdot\vec{n},v\rangle_{H^{-1/2},H^{1/2}}$ for all $w\in H^1(\Omega;\div)$ and all $v\in C^1(\overline{\Omega})$.

\subsection{Boundary conditions on $\mathrm{DM}(\Omega)$ and $\mathrm{M}^p(\Omega;\div )$} \label{sec:BC}		
The map $\mathrm{N}$ allows us to define subspaces of $\mathrm{DM}(\Omega)$ and $\mathrm{M}^p(\Omega;\div )$ of { vector} measures whose normal traces vanish (in the sense described by $\mathrm{N}$) on a part $\Gamma$ of the boundary {  $\partial\Omega$ (if it exists)}  as we see in what follows.  In this vein, consider the following:
	
	\begin{definition}
		Let $\Gamma \subset \partial \Omega$ be non-empty, and define
		\begin{equation*}
		\mathrm{DM}_{\,\Gamma}(\Omega)=\{ \mu \in \mathrm{DM}(\Omega) : 
		{ \mathrm{N}(\phi,\mu)} = 0  
		\enspace 
		\text{for all}
		\enspace
		\phi \in { \mathbf{C}_b^1}(\overline{\Omega}) 
		\enspace
		\text{such that}
		\enspace
		\phi \vert_{ \overline{\partial \Omega \setminus \Gamma}} = 0 \}, 
	\end{equation*} 	
and analogously we define $\mathrm{M}_{\,\Gamma}^{\,p}(\Omega;\div )$ for $1\leq p\leq +\infty$ as
		\begin{equation*}
		\mathrm{M}_{\,\Gamma}^{\,p}(\Omega;\div )=\{ \mu \in \mathrm{M}^{\,p}(\Omega;\div ) : 
		{ \mathrm{N}(\phi,\mu)} = 0  
		\enspace 
		\text{for all}
		\enspace
		\phi \in { \mathbf{C}_b^1}(\overline{\Omega}) 
		\enspace
		\text{such that}
		\enspace
		\phi \vert_{ \overline{\partial \Omega \setminus \Gamma}} = 0 \}.
	\end{equation*} 	
If $\Gamma = \emptyset$,
we denote $\mathrm{M}_{\Gamma}^p(\Omega ; \mathrm{div})$ as $\mathrm{M}^p(\Omega ; \mathrm{div})$
and $\mathrm{DM}_{\Gamma}(\Omega)$ as $\mathrm{DM}(\Omega)$.
	\end{definition}

	  Note that due to Whitney's extension result there are always non-trivial $\phi$ functions in $C^k$ with arbitrary $k\in\mathbb{N}$ and such that they vanish exactly in the closure of $\partial \Omega \setminus \Gamma$. It follows that $\mathrm{DM}_{\,\Gamma}(\Omega)$ and $\mathrm{M}_{\,\Gamma}^{\,p}(\Omega;\div )$ are linear subspaces of $\mathrm{DM}(\Omega)$ and $\mathrm{M}^p(\Omega;\div )$, respectively. In addition, if $\mu_n\to\mu$ in $\mathrm{DM}(\Omega)$, that is $\mu_n\to\mu \in \mathrm{M}(\Omega)^N$ and $ \mathrm{div}\, \mu_n\to \mathrm{div}\, \mu \in \mathrm{M}(\Omega)$, and $\phi \in { \mathbf{C}_b^1}(\overline{\Omega})$, then
\begin{equation*}
	\int_{\Omega} \nabla \phi \cdot \dif\mu_n+\int_\Omega \phi \dif\: \div\:\mu_n \to \int_{\Omega} \nabla \phi \cdot \dif\mu+\int_\Omega \phi \dif\: \div\:\mu ,
\end{equation*} 
so $\mathrm{DM}_{\,\Gamma}(\Omega)$ is closed with respect to the $\mathrm{DM}(\Omega)$ norm. Analogously,  $\mathrm{M}_{\,\Gamma}^{\,p}(\Omega;\div )$ is closed in $\mathrm{M}^{\,p}(\Omega;\div )$.

The simplest example of a measure $\mu$ in $\mathrm{M}^{\,p}(\Omega;\div )$, and hence also in $\mathrm{DM}_{\,\Gamma}(\Omega)$, is when $\mu$ given by $N$-copies of the $N$-dimensional Hausdorff measure $\mathcal{H}^N$. Clearly, $\mu = (\mathcal{H}^N , ... , \mathcal{H}^N)$ belongs to $\mathrm{M}(\Omega)^N$ and for $\Omega$ sufficiently regular we have 
	\begin{equation*}
		\int_{ \Omega } \nabla \phi \cdot \mathrm{d} \mu  = 0 
	\end{equation*}
	for every $\phi \in C_c^{\infty}(\Omega)$ {  by direct integration by parts}. Hence, 
	$\mathrm{div}\, \mu = 0 $ and $\mu \in \mathrm{M}(\Omega;\div \:0)\subset \mathrm{M}^p(\Omega;\div )$, and if $\Gamma \subset \partial \Omega$ is non-empty then it follows that, in general, $\mu \notin \mathrm{M}_{\,\Gamma}^{\,p}(\Omega;\div )$.

\begin{figure}[h]
	   	\includegraphics[scale=0.235]{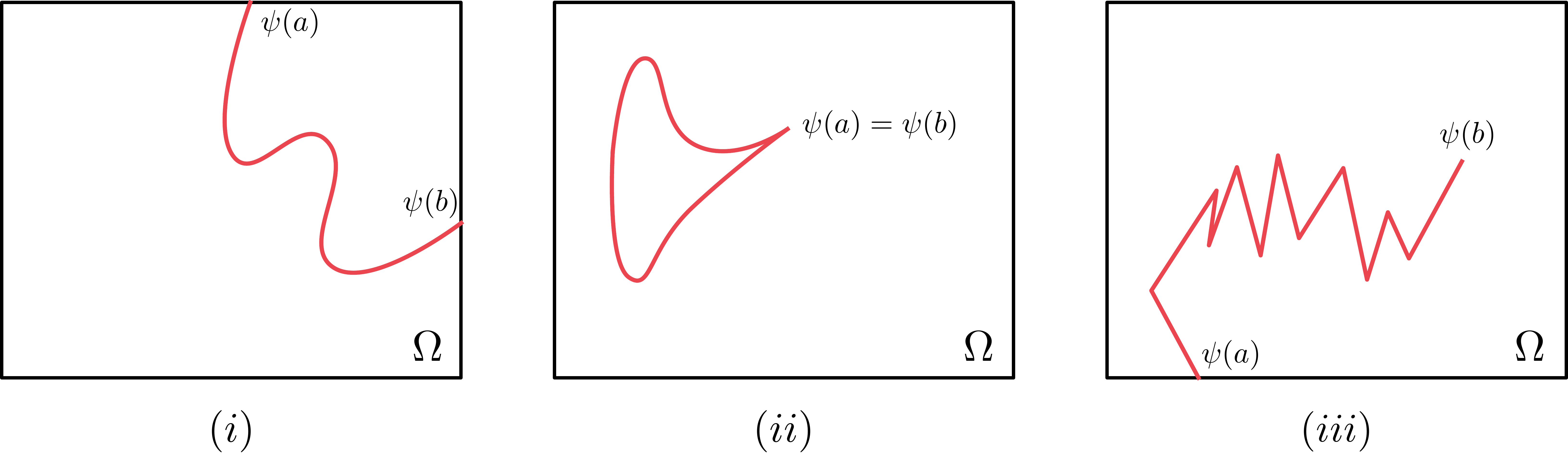}
  	\caption{Possible $C$ piecewise smooth curves determined by $\psi:(a,b)\to C$ and associated the measure $\mu=\psi^{\prime} \circ \psi^{-1}
			\, \mathcal{H}^1 \Lcorner C$  from \Cref{meas1}. In $(i)$, the endpoints $\psi(a)$, and $\psi(b)$ are located at the boundary $\partial\Omega$ of $\Omega$ so that $\mu\in\mathrm{M}^p(\Omega;\div )$ for every $p$, and this also holds true for $(ii)$ where the endpoints are  the same point within $\Omega$. Finally, in $(iii)$ the endpoints are different and $\psi(b)$ is located in $\Omega$ so that $\mu$ as no divergence represented as a function.}
  	  \label{fig:Ex1}
  	\end{figure}

The following example establishes that for some measures determined by piece-wise regular curves, divergences exist and are either zero or the difference of point measures.

\begin{example} \label{meas1}
	Let $a<b$ and 
	$\psi : (a,b)  \to C\subset \Omega$ be a continuously 
	differentiable bijection. Suppose that $\psi'$ is never zero and is integrable over $(a,b)$ so that $C$ is a regular rectifiable curve. { Assume} that $C$ is parametrized by arc length, so that $|\psi'(t)|=1$ for all $t\in(a,b)$ and $b-a$ is the length of the curve $C$. In addition, we assume that $\psi$ is extended to $[a,b]$  with $\psi(a)$ and $\psi(b)$ on $\overline{\Omega}$ so that the endpoints of $C$ may lie on the boundary $\partial\Omega$.  For $B\in \mathcal{B}(\Omega)$, define the set-function 
		\begin{equation} \label{measvec}
			\mu ( B ) = \int_{ B\cap C} \psi^{\prime} \circ \psi^{-1}
			\, \mathrm{d} \mathcal{H}^1.
		\end{equation}
	 Note that $\mu\in\mathrm{M}(\Omega)^{N}$ and 
	 \begin{equation*}
	 	|\mu|(\Omega)=\int_{C}\mathrm{d} \mathcal{H}^1=b-a.
	 \end{equation*}
	For $\phi \in  { \mathbf{C}_b^1}(\overline{\Omega})$, by a change of variables of integration:
	\begin{align*} 
			\int_{\Omega} \nabla \phi \cdot \, \mathrm{d}\mu &= 
			\int_{C} \nabla \phi \cdot \psi^{\prime} \circ \psi^{-1} \, 
			\mathrm{d} \mathcal{H}^{1}=\int_a^b\nabla \phi (\psi(t)) \cdot \psi^{\prime} ( \psi^{-1}(\psi(t))|\psi'(t)|\dif t.
	\end{align*}	
	Since $|\psi'(t)|=1$ for all $t\in(a,b)$, the integrand on the right hand side is $\frac{d}{dt} \phi(\psi(t))$. Hence
		\begin{equation} \label{eq:boundary_example}
			\int_{\Omega} \nabla \phi \cdot \, \mathrm{d}\mu = \phi(\psi(b^-)) - \phi(\psi(a^+)).
		\end{equation}
	In fact, it is not hard to see that the above holds true for $\psi : (a,b)  \to C\subset \Omega$ continuous, bijective and piecewise continuously differentiable with $|\psi'(t)|=1$ wherever the derivative exists.		
	The locations of $\psi(a)$ and $\psi(b)$ on $\overline{\Omega}$ lead to different scenarios as we next explore.
	
	Suppose that $\psi(a) , \psi(b) \in \partial \Omega$. It follows that \eqref{eq:boundary_example} is identically zero for all $\phi \in C_c^{\infty}(\Omega)$.
	Thus, $\div \mu =0$ and hence $\mu$ belongs to $\mathrm{M}^p(\Omega;\div )$ for all $p$. Further, $\div \mu =0$ also in the case that $\psi(a)=\psi(b)$ even if the point is not in 
	$\partial \Omega$. However, if $\psi(a)$ or $\psi(b)$ are not in $\partial\Omega$ and are not identical, then $\mu\notin \mathrm{M}^p(\Omega;\div )$ for each $p$: for, in general, if $\psi(a),\psi(b)\in\Omega$, then from  \eqref{eq:boundary_example}  and definition of divergence
	\begin{equation}\label{eq:divdelta}
		\div \mu=\delta_{\psi(a)}-\delta_{\psi(b)},
	\end{equation} 
	that is, the difference of two Dirac deltas at the points $\psi(a),$ and $\psi(b)$.
	
	If $\psi(a) , \psi(b) \in \overline{\partial\Omega\setminus\Gamma}$, then \cref{eq:boundary_example} vanishes
        for all $\phi \in { \mathbf{C}_b^1}(\overline{\Omega})$ such that $\phi \vert_{ \overline{\partial \Omega \setminus \Gamma}} = 0$. Hence $\mathrm{N}{ (\phi , \mu )} = 0$
        which gives $\mu\in \mathrm{M}_{\,\Gamma}^{\,p}(\Omega;\div )$.	 However, if $\psi(a) \in \overline{\partial\Omega\setminus\Gamma}$ and $\psi(b) \in \partial\Omega$ but $\psi(b) \notin  \overline{\partial\Omega\setminus\Gamma}$,
	it is clear that \cref{eq:boundary_example} fails to vanish for 
	some $\phi$, in which case 
	$\mathrm{N}{ (\phi , \mu )} \neq 0 $ so that 
	$\mu \notin \mathrm{M}_{\,\Gamma}^{\,p}(\Omega;\div )$. 
\end{example}	

The next example extends the previous one, and shows that for a point-wise weighted Hausdorff $\mathcal{H}^{1}$ measure { restricted to a} piece-wise regular curve, the divergence contains in general an $\mathcal{H}^{1}$-weighted term in addition to the difference of Dirac deltas as in \eqref{eq:divdelta}.

\begin{example} \label{example:crack0}
	Consider $a<b$ and 
	$\psi : (a,b)  \to C\subset \Omega$ defined as in the previous example. Further, let $h:[a,b]\to \mathbb{R}$ be continuously differentiable and define 
	\begin{equation*}
		\mu= (h\psi^{\prime}) \circ \psi^{-1}\:\mathcal{H}^1\Lcorner C.
	\end{equation*}
	Hence, if $\phi \in C^{1}(\overline{\Omega})$ then similarly with the previous example we observe that
	\begin{align*} 
			\int_{\Omega} \nabla \phi \cdot \, \mathrm{d}\mu &= 
		    \int_a^bh(t)\frac{d}{dt} \phi(\psi(t)) \dif t=h(b)\phi(\psi(b)) - h(a)\phi(\psi(a))-\int_a^bh'(t)\phi(\psi(t))\dif t,
	\end{align*}	
	{ (recall that $|\psi'(t)|=1$)} or equivalently
		\begin{align*} 
			\int_{\Omega} \nabla \phi \cdot \, \mathrm{d}\mu &= \int_\Omega (h\circ \psi^{-1})\phi\dif \delta_{\psi(b)}-\int_\Omega (h\circ \psi^{-1})\phi\dif\delta_{\psi(a)}-\int_{C}(h'\circ \psi^{-1})\phi \dif \mathcal{H}^1.
	\end{align*}	
	It follows that $\mu\in\mathrm{DM}(\Omega)$ with
	\begin{equation*}
		\div \mu = h\circ \psi^{-1} \delta_{\psi(a)}-h\circ \psi^{-1} \delta_{\psi(b)}+ h'\circ \psi^{-1}\:\mathcal{H}^1\Lcorner C.
	\end{equation*}
	In general, $\mu$ does not belong to $\mathrm{M}^p(\Omega;\div )$ for any $p$, unless $h$ is a constant  {  and additional conditions for $\psi$ are met}.

	\begin{figure}[h]
	   	\includegraphics[scale=0.33]{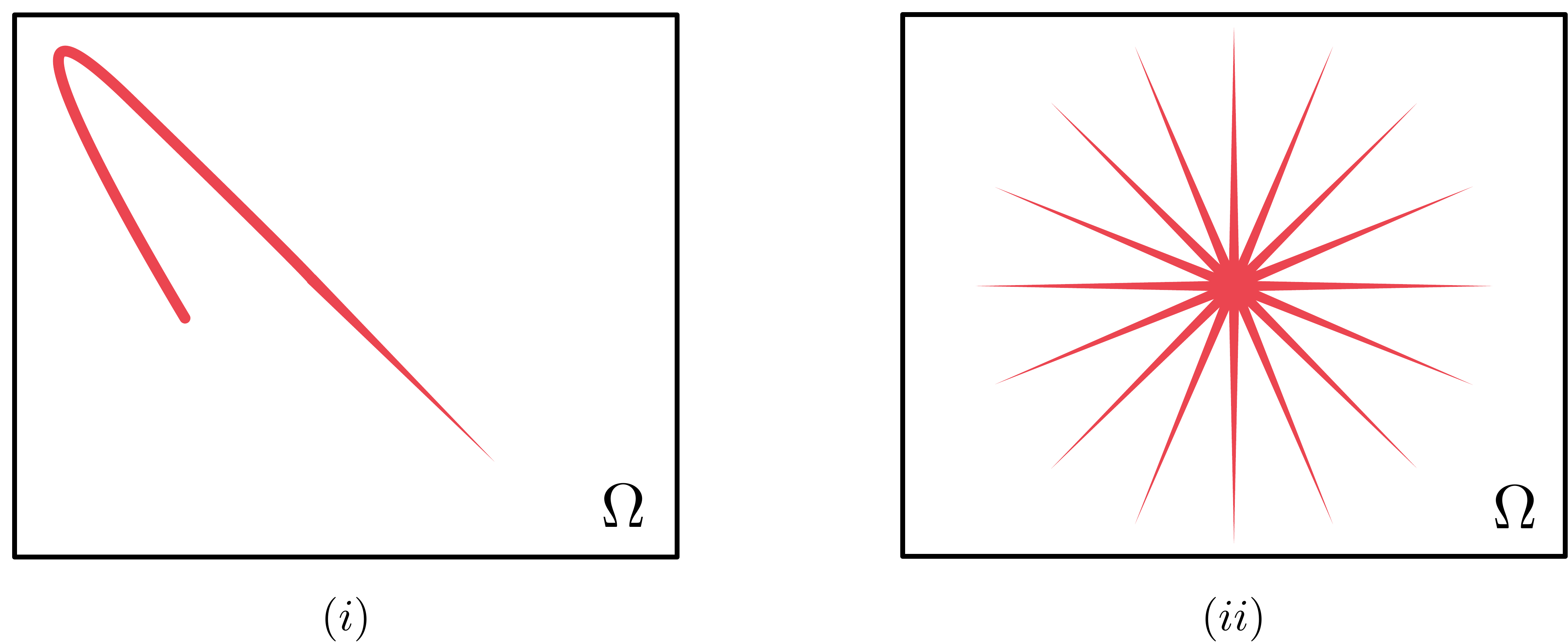}
  	\caption{Examples of supports of measures for \Cref{example:crack0}; the width of the curve corresponds to the magnitude of $|h (\cdot)|$ at each point. In $(i)$, the associated measure $\mu$ possesses a divergence given by the sum of a weighted Dirac delta and a weighted Hausdorff measure $\mathcal{H}^{1}$ on $C$. In $(ii)$, the measure associated to the graph possesses a divergence that is a finite sum of weighted Hausdorff measures $\mathcal{H}^{1}$.}
  	  \label{fig:graph}
  	\end{figure}
\end{example}

\subsection{Divergence zero measures associated with binary trees}\label{sec:BC2}

We consider in this section measures induced by binary trees and combinations thereof. This geometrical structure allows one to define measures with zero divergence as infinite series of weighted-$\mathcal{H}^{1}$ measures.

	\subsubsection{Trees with base points in { $\overline{\Omega}$}} \label{example:crack1} Let $\Omega\subset \mathbb{R}^N$ be an open set with $N\geq 2$, {  and where $\partial\Omega$ is not empty}. Consider the countable collection of non-intersecting open line segments $\{L_i\}$ in $\Omega$ such that $L_0$ has an endpoint $x_0^0$ in { $\overline{\Omega}$}, the other endpoint $x_1^0$ is shared as endpoint with only other two segments and so on so that 
	the collection of segments forms a binary tree $L$ (see \Cref{fig:graph}), i.e., 
	\begin{equation*}
		L=\bigcup_{i=0}^\infty L_i.
	\end{equation*} 
	Further, we assume that $\sum_i |L_i|<+\infty$, and  that the order of the segments is such that $L_1,$ and $L_2$ share an endpoint with $L_0$, then $L_3$ and $L_4$ share an endpoint with $L_1$ and $L_5$ and $L_6$ share an endpoint point with $L_2$, and so on and so forth. Denote the endpoints of the segment $L_i$ as $x_0^i$ and $x_1^i$; the 1 subscript denotes a shared endpoint with two segments $L_j$ and $L_k$ such that $i<j,k$.  {  Finally, we assume that all branches approach the boundary, that is $\mathrm{dist}(L_i,\partial\Omega)\to 0$ as $i\to\infty$.}
	
	{  Suppose that each $L_i$ is parameterized 
	by $\psi_i : [0, |L_i| ] \to \Omega$ given by 
	\begin{equation*}
		\psi_i( t ) = \frac{|L_i|-t}{|L_i|} x_0^i + \frac{t}{|L_i|}  x_1^i
	\end{equation*}
	so that each $|\psi_i^{\prime}| = 1$.} Further, consider the sequence of vectors $\{e_i\}$ in $\mathbb{R}^N$ defined as
	\begin{equation*}
		e_i:=\frac{x_1^i-x_0^i}{|x_1^i-x_0^i|} = \psi_i^{\prime},
	\end{equation*}
	that is, $e_i$ is the vector of the line that contains $L_i$ . Let $\{h_i\}$ be the sequence of real numbers defined as
	\begin{equation*}
		h_i=2^{-k} \qquad \text{if }\: 2^{k}-1\leq  i \leq 2^{k+1}-2, \qquad \text{ for }\: k=0,1,2,\ldots  
	\end{equation*}
	so that the sequence $\{h_i\}$ has one $1/1$ entry,  two $1/2$ entries, 
	four $1/4$ entries, eight $1/8$ entries and so on. 
	{  For $k=0,1,2,\ldots$, define 
	$\mu^k: \mathcal{B}(\Omega)\to \mathbb{R}^N$ as
	\begin{equation*}
		\mu^k=\sum_{i=0}^{r(k)-1}\mu_i  \qquad \text{ where } \quad \mu_i= h_ie_i \mathcal{H}^1\Lcorner L_i \quad \text{ with } \quad r(k)=\sum_{i=0}^k 2^i.
	\end{equation*}
	Since each $\mu_i$ takes the form of the 
	vector-valued measure given in \eqref{measvec},
	we have $\mu^k \in \mathrm{M}(\Omega)^N$. Further, since $|\mu_i|=|L_i|$ and $\sum_i |L_i|<+\infty$, we have that $\mu^k\to \mu$ as $k\to\infty$ to some $\mu\in \mathrm{M}(\Omega)^N$ such that}

	\begin{equation*}
	\mu=\sum_{i=0}^\infty\mu_i   \qquad \text{ where }\quad  \mu_i= h_ie_i \mathcal{H}^1\Lcorner L_i \quad  \text{ for } \quad  i=0,1,2,\ldots.
\end{equation*}
	
	 { A few words are in order concerning $\{\mu^k\}$, note that $\mu^0$ is associated with the trunk of the tree, and  that $\mu^k$ for $k>0$ contains $2^k$ more terms (branches of the tree) than $\mu^{k-1}$. 
	
	If $k = 0$,
	we apply \eqref{eq:boundary_example} to 
	obtain for $\phi \in \mathbf{C}_b^1(\overline{\Omega})$ that
	\begin{equation*}
		\int_{\Omega} \nabla \phi \cdot \mathrm{d} \mu^0
		= \phi(x_1^0) - \phi(x_0^0). 
	\end{equation*}
	Notice now that for $k\in\mathbb{N}$,
	repeated application of \eqref{measvec}
	leads to cancellation of all ``intermediate"
	nodes in the binary tree in the sense that
	\begin{equation} \label{finite_tree}
		\int_{\Omega} \nabla \phi \cdot \mathrm{d} \mu^k =- \phi(x_0^0) + 
		2^{-k}\sum_{ i=2^{k}-1}^{2^{k+1}-2} \phi(x_i^1). 
	\end{equation}
	Combining \eqref{eq:divdelta} with the expression above gives an 
	expression for $\mathrm{div} \, \mu^k$ as the (weighted) family 
	of point masses:
	\begin{equation*}
	\mathrm{div} \mu^k = \delta_{x_0^0} -  
		 2^{-k}\sum_{ i=2^{k}-1}^{2^{k+1}-2} \delta_{x_i^1}.
	\end{equation*}

Let $\phi \in C_c^{\infty}(\Omega)$, where $\mathrm{supp}(\phi)\subset K\subset \Omega$ and $K$ is compact. Since $\mathrm{dist}(L_i,\partial\Omega)\to 0$ as $i\to\infty$, then there exists a sufficiently large $I\in\mathbb{N}$ such that  $\phi|_{L_i}=0$ for $i\geq I$. In particular, this means that for $2^k-1\geq I$, we observe
	\begin{equation*}  
		 \sum_{ i=2^{k}-1}^{2^{k+1}-2} \phi(x_i^1) =0.
	\end{equation*}
Therefore, for an arbitrary $\phi \in C_c^{\infty}(\Omega)$, we have from \eqref{finite_tree} that
	\begin{equation}\label{eq:boundary_value}
		\int_{\Omega} \nabla \phi \cdot \mathrm{d} \mu=   - \phi(x_0^0).
	\end{equation}
	We conclude that $\mathrm{div} \, \mu=-\delta_{x_0^0}$. If in addition we have that $x_0^0 \in \partial \Omega$,
	 then $\mathrm{div} \, \mu=0$.}  Therefore, $\mu \in \mathrm{M}^p(\Omega ; \mathrm{div})$. Interestingly, the inclusion 
of $\mu$ in $\mathrm{M}_{\Gamma}^p(\Omega ; \mathrm{div})$ depends upon the location of $\Gamma$ relative to $L$. Observe that
if $x_0^0 \in \overline{\partial \Omega \setminus \Gamma}$ then, \cref{eq:boundary_value} is identically zero for all $\phi\in { \mathbf{C}_b^1}(\overline{\Omega})$ such that $\phi \vert_{ \overline{\partial \Omega \setminus \Gamma}} = 0$,
which gives ${ \mathrm{N}(\phi,\mu)} = 0 $ so that 
$\mu \in \mathrm{M}_{\Gamma}^p(\Omega ; \mathrm{div})$.

	\begin{figure}[h]
	   	\includegraphics[scale=0.38]{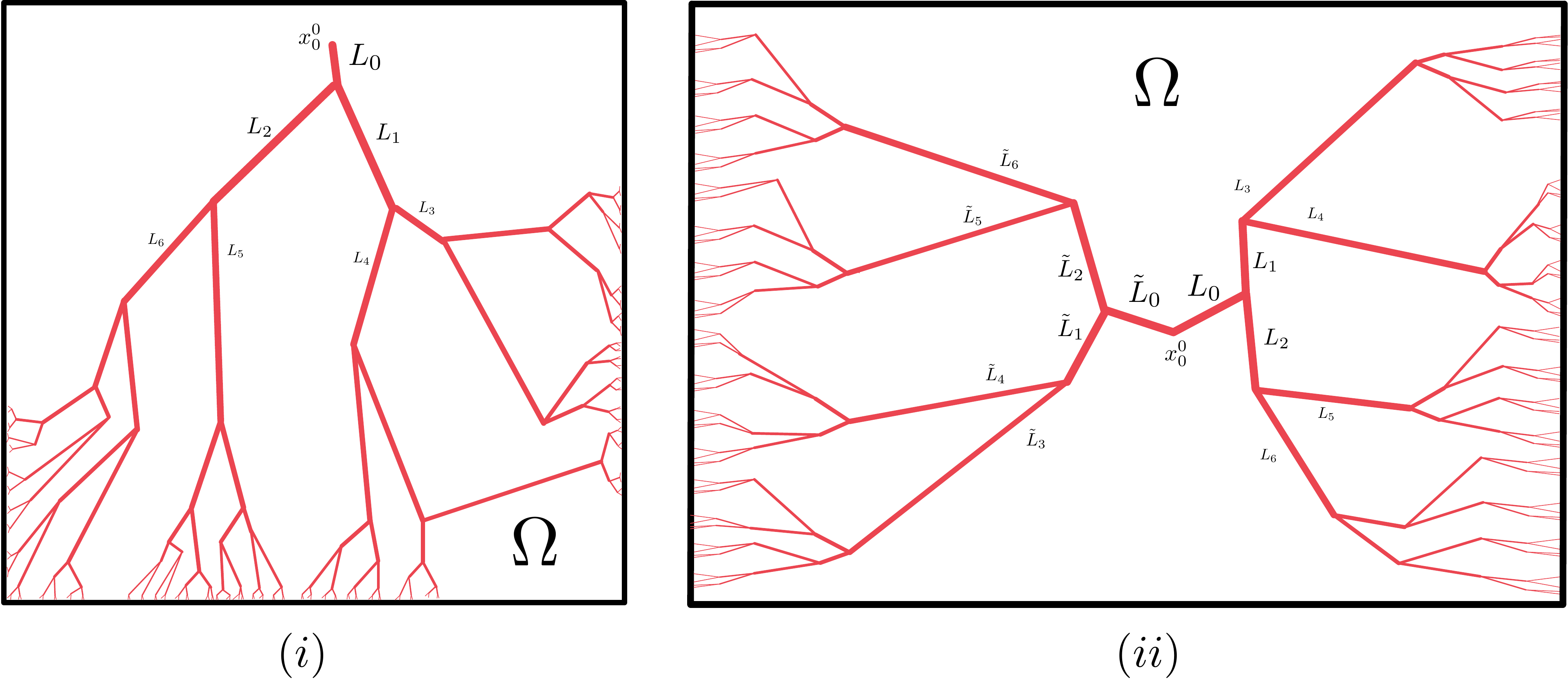}
  	\caption{Measures generated by binary trees where the width of the branch corresponds to the weight $h_i$ of the associated $\mathcal{H}^1$-measure. Tree with base point in $x_0^0\in \Omega$ in $(i)$, and two binary trees sharing the base point in  $(ii)$.}
  	  \label{fig:graph33}
  	\end{figure}
 
\subsubsection{Two trees sharing the base point}
Consider now two binary trees $L$ and $\tilde{L}$ composed of segments $\{L_i\}$ and $\{\tilde{L}_i\}$, respectively,  as defined in the previous section. Suppose $L$ and $\tilde{L}$ share the same base point $x_0^0\in \Omega$, but otherwise are non-intersecting and the closure of the union of all segments is in $\Omega$, i.e., 
\begin{equation*}
	\overline{L}\cap\overline{\tilde{L}}=\{x_0^0\}, \qquad \text{and}\qquad \overline{L}\cup\overline{\tilde{L}}\subset \Omega;
\end{equation*}
 see $(ii)$ in \Cref{fig:graph33}. Let 
 \begin{equation*}
 	W=\left(\bigcup_{i  =0 }^{\infty} L_i \right) \bigcup \left(\bigcup_{i  =0 }^{\infty} \tilde{L}_i\right),
 \end{equation*}
 and let $\mu$ and $\tilde{\mu}$ be the measures associated to $L$ and $\tilde{L}$, respectively, as constructed in \Cref{example:crack1}. Define the Borel measure $\xi: \mathcal{B}(\Omega)\to \mathbb{R}^N$ as
 \begin{equation*}
 	\xi=\mu -\tilde{\mu}, 
 \end{equation*}
and note that for any $\phi\in C^1(\overline{\Omega})$, observe that 
	\begin{equation}
		\int_\Omega \nabla \phi\cdot \dif \xi=-\phi(x_0^0)+\phi(x_0^0)=0.
	\end{equation}
	Hence, it is not only the case that $\mathrm{div}\, \xi = 0$, but 
	also that $\mathrm{N} { ( \phi ,\xi  )} = 0 $ 
	for those $\phi \in { \mathbf{C}_b^1}(\overline{\Omega})$ so that $\xi \in \mathrm{M}_{\Gamma}^p(\Omega ; \mathrm{div})$ 
	regardless of the location of  $\Gamma$.

\section{Bounded sets of measures and convergence }\label{sec:sets}

{
In this section, we consider { both} some initial results of the natural order in $\mathrm{M}(\Omega)$ induced by the cone $\mathrm{M}^{+}(\Omega)$ and some convergence results {  for sequences of convex sets of the type $\{ \mu \in X : { |\mu|} \le \alpha_n \},$ where $\{\alpha_n\}$ is a sequence in $\mathrm{M}^{+}(\Omega)$ and  $X$ is one of the spaces of measures of interest:  $\mathrm{M}(\Omega)^N$, $\mathrm{DM}_{\Gamma}(\Omega)$, or  $\mathrm{M}_{\Gamma}^p(\Omega; \mathrm{div})$}. The latter arise in applications to optimization problems over subsets of  $\mathrm{M}(\Omega)^N$ that are bounded  with respect to some specific measure. 

\begin{definition}
	Let $\mu$ and $\sigma$ be elements of $\mathrm{M}(\Omega)$. We write $\mu\leq \sigma$ if
	\begin{equation*}
		\sigma-\mu \:\text{ belongs to }\: \mathrm{M}^+(\Omega).
	\end{equation*}
\end{definition}
Note that $\mu\leq\sigma$ is equivalent to the requirement that $\mu(B)\leq\sigma(B)$ for all $B\in\mathcal{B}(\Omega)$ or that 
\begin{equation*}
	\int_\Omega w\dif \mu \leq \int_\Omega w\dif \sigma,
\end{equation*}
for all $w\in C_c(\Omega)$, or $w\in C_0(\Omega)$, such that $w(x)\geq 0$ for all $x\in\Omega$; {  the latter follows by a combination of the dominated convergence theorem and the fact that step functions are dense in $C_c(\Omega)$ and $C_0(\Omega)$}}. 
We further have the following equivalency among orders for non-negative measures. 

\begin{proposition} \label{defn_bounded}
	Let $\mu,\sigma\in\mathrm{M}^+(\Omega)$ then the following are equivalent
	\begin{itemize}
		\item[$\mathrm{(a)}$] $\mu\leq\sigma$.
		\item[$\mathrm{(b)}$] $\mu(O)\leq\sigma(O)$ for all open sets $O$ such that $O\subset \Omega$.
		\item[$\mathrm{(c)}$] $\mu(K)\leq\sigma(K)$ for all compact sets $K$ such that $K\subset \Omega$.
		\item[$\mathrm{(d)}$] For all  non-negative $w\in C_b^\infty(\Omega)$, it holds true that
		\begin{equation}\label{eq:ineqw}
			\int_\Omega w\dif \mu \leq \int_\Omega w\dif \sigma.
		\end{equation}

	\end{itemize}
\end{proposition}

\begin{proof}
		The directions (a)$\Rightarrow$(b) and (a)$\Rightarrow$(c) are trivial. The proof for (b)$\Rightarrow$(a) and (c)$\Rightarrow$(a) follows because non-negative elements of $\mathrm{M}(\Omega)$ are inner and outer regular: for $\alpha\in\mathrm{M}^+(\Omega)$ and $B\in\mathcal{B}(\Omega)$ we have
		\begin{equation*}
			\alpha(B)=\inf\{\alpha(O):O \text{ open \& } B\subset O \subset\Omega\}=\sup\{\alpha(K):K \text{ compact \& } K\subset B \}.
		\end{equation*}
	Then (b)$\Rightarrow$(a) follows by taking the $\inf$ over all open sets $O$ in $\Omega$ containing $B$  and (c)$\Rightarrow$(a) by taking the $\sup$ over all compact sets $K\subset \Omega$, respectively.  
	
	In order to prove that (a)$\Leftrightarrow$(d), we only need to prove that to consider $C_b^\infty(\Omega)$ is equivalent to considering $C_c(\Omega)$ as the test function space in (d). Suppose that (d) holds true and let $w\in C_c(\Omega)$ be arbitrary. Then via classical mollifier techniques ({  see Chapter 2 in  \cite{adams_fournier_2003}}), there exists a sequence $\{w_n\}$ such that $w_n\in C_c^\infty(\Omega)$ such that $w_n\to w$ uniformly so that
	\begin{equation*}
		\int_\Omega w_n \dif \eta \to \int_\Omega w \dif \eta,
	\end{equation*}
as $n\to\infty$ for arbitrary $\eta\in\mathrm{M}^+(\Omega)$ so it is direct to prove that (d)$\Rightarrow$(a). Conversely, suppose \eqref{eq:ineqw} holds true for all non-negative functions in $C_c(\Omega)$ and let $w\in C_b^\infty(\Omega)$ be non-negative and  arbitrary. {  Note that } there exist a sequence of compact sets $\{K_n\}$ such that $\mu(\Omega\setminus K_n),\sigma(\Omega\setminus K_n)\to 0$ as $n\to\infty$. { Then,} let $w_n:\Omega\to\mathbb{R}$ such that $w_n(x)=w(x)$ for $x\in K_n$, each $w_n$ has a compact support $\tilde{K}_n$ that contains $K_n$, and $0\leq w_n\leq w$. Hence,
\begin{align*}
	\int_\Omega w \dif \mu &=\int_{\Omega} w_n \dif \mu+\int_{\Omega\setminus K_n} (w-w_n) \dif \mu\\
	&\leq \int_{\Omega} w_n \dif \sigma +\int_{\Omega\setminus K_n} (w-w_n) \dif \mu\\
	&{ =} \int_{\Omega} w \dif \sigma +\int_{\Omega\setminus K_n} (w_n-w) \dif \sigma+\int_{\Omega\setminus K_n} (w-w_n) \dif \mu.
\end{align*}
Since
\begin{equation*}
	\left|\int_{\Omega\setminus K_n} (w_n-w) \dif \sigma+\int_{\Omega\setminus K_n} (w-w_n) \dif \mu\right| \leq 2\left(\sup_{x\in\Omega}|w(x)|\right)(\sigma(\Omega\setminus K_n)+\mu(\Omega\setminus K_n))\to 0,
\end{equation*}
as $n\to \infty$, {  we have proven that} (a)$\Rightarrow$(d).
\end{proof}

We consider now sets of measures whose total variation is dominated by a non-negative measure. Specifically, let $\alpha \in \mathrm{M}^{+}(\Omega)$ be arbitrary, and define the set $\mathbf{K}(\alpha; X)$ as 
		\begin{equation}\label{eq:Kgeneral}
			\mathbf{K}(\alpha; X) := \{ \mu \in X : |\mu| \le \alpha \}, 
		\end{equation}
where $X$ is one of the spaces of measures of interest  $\mathrm{M}(\Omega)^N$, $\mathrm{DM}_{\Gamma}(\Omega)$, or  $\mathrm{M}_{\Gamma}^p(\Omega; \mathrm{div})$. Note that $\mathbf{K}(\alpha; X)$ is (in all cases) convex, closed, and non-empty since $0 \in \mathbf{K}(\alpha; X)$. The main focus of the rest of the paper is to study properties of the map
\begin{equation*}
\alpha\mapsto \mathbf{K}(\alpha; X),
\end{equation*}
that are useful for the study of stability of optimization problems and other applications where moving sets are required.

We consider two different kinds of results, a \emph{forward} and a \emph{backward} kind: Consider a sequence $\{\alpha_n\}$ in $\mathrm{M}^{+}(\Omega)$ converging to some $\alpha^*\in \mathrm{M}^{+}(\Omega)$ in some sense, then 

\begin{itemize}
	\item[i.] Suppose the sequence $\{\mu_n\}$ is such that $\mu_n\in \mathbf{K}(\alpha_n; X)$ for $n\in \mathbb{N}$. Is there a subsequence of  $\{\mu_n\}$ converging wth respect to some topology to $\mu^*\in \mathbf{K}(\alpha^*; X)$?
	\item[ii.] Suppose that $\tilde{\mu}\in \mathbf{K}(\alpha^*; X)$ is arbitrary. Is there a sequence $\{\tilde{\mu}_n\}$ such that $\tilde{\mu}_n\in \mathbf{K}(\alpha_n; X)$ for $n\in\mathbb{N}$ and such that it converges with respect to some topology to $\tilde{\mu}$?
\end{itemize}
Such form of set convergence  (when topologies are chosen properly) is called \emph{Mosco convergence} \cite{mosco1969convergence,mosco1967approximation} and they can be also described by means of the more classical Painlev\'{e}-Kuratowski set limits \cite{kuratowski1948topologie}. See the monograph \cite{Aubin2009} for an historical account on the notions of set convergence, and \cite{menaldi2021some} for relation to Gamma convergence. In general, the construction of the sequence  $\{\tilde{\mu}_n\}$, called \emph{recovery sequence} in ii is a more complicated task than the one in i.

Using the notation of \eqref{eq:Kgeneral}, for rest of the paper we use

\begin{equation} \label{KaGammaDiv}
	\mathbf{K}(\alpha):=\mathbf{K}\big(\alpha; \mathrm{M}(\Omega)^N\big) \qquad\text{\and}\qquad \mathbf{K}^p_{\Gamma}(\alpha; \mathrm{div}):=\mathbf{K}\big(\alpha; \mathrm{M}_{\Gamma}^p(\Omega; \mathrm{div})\big) 
\end{equation}
to describe subsets of $\mathrm{M}(\Omega)^N$ and $\mathrm{M}_{\Gamma}^p(\Omega; \mathrm{div})$ with total variation bounded by $\alpha\in \mathrm{M}^{+}(\Omega)$.   As usual, if $\Gamma=\emptyset$, then we denote $\mathbf{K}^p_{\Gamma}(\alpha; \mathrm{div})$ as $\mathbf{K}^p_{}(\alpha; \mathrm{div})$.

\subsection{Forward results}\label{sec:M1}

The following lemma establishes that weak convergence of the sequence $\{\alpha_n\}$ of upper bounds  is stable in the sense that any sequence $\{\mu_n\}$ such that  $\mu_n\in \mathbf{K}(\alpha_n)$ admits a convergent subsequence with limit point in $\mathbf{K}(\alpha)$ and where $\alpha$ is the weak limit of $\{\alpha_n\}$.
	
\begin{lemma} \label{lem1}
	Suppose that $\{\alpha_n\}$ is a sequence in $\mathrm{M}^+(\Omega)$ such that $\alpha_n\rightharpoonup \alpha$ in $\mathrm{M}(\Omega)$ for some $\alpha$, and let $\{\mu_n\}$ be a sequence in $\mathrm{M}(\Omega)^N$ such that $\mu_n\in \mathbf{K}(\alpha_n)$ for $n\in \mathbb{N}$. Then there exists a  subsequence of $\{\mu_n\}$ weakly convergent to some $\mu\in \mathbf{K}(\alpha)$ in $\mathrm{M}(\Omega)^N$.
\end{lemma}

\begin{proof}
	Since $\alpha_n\rightharpoonup \alpha$ in $\mathrm{M}(\Omega)$, there exists an $M>0$ such that $\|\alpha_{  n}\|_{\mathrm{M}(\Omega)}=\alpha_n(\Omega)\leq M$ for all $n\in \mathbb{N}$ by the uniform boundedness principle.  Since $\mu_n\in \mathbf{K}(\alpha_n)$ for $n\in \mathbb{N}$, then $|\mu_n|(\Omega)\leq M$ for all $n\in \mathbb{N}$ as well. Thus, $\{\mu_n\}$ and $\{|\mu_n|\}$ are bounded in $\mathrm{M}(\Omega)^N$ and $\mathrm{M}(\Omega)$, respectively. Hence, there exist subsequences $\{\mu_{n_j}\}$ and $\{|\mu_{n_j}|\}$ such that 
	\begin{equation}
		\mu_{n_j}\rightharpoonup \mu \qquad \text{ and} \qquad |\mu_{n_j}|\rightharpoonup \sigma,
	\end{equation}
	for some $\mu\in \mathrm{M}(\Omega)^N$ and some $\sigma\in \mathrm{M}^+(\Omega)$  by \cite[Proposition 4.2.2]{abm14}. Further, by \cite[Corollary 4.2.1]{abm14} we have
	\begin{equation}\label{eq:|mu|leqa}
		|\mu|\leq \sigma.
	\end{equation}
	Since $\mu_{n_j}\in \mathbf{K}(\alpha_{n_j})$, we have 
		\begin{equation*}
\int_{\Omega} f \, \mathrm{d} |\mu_{n_j}| \le \int_{\Omega} f \, \mathrm{d} \alpha_{n_j}, 
\end{equation*}
for all $f\in C_c(\Omega)$ such that $f\geq 0$.
Given that $\alpha_{n_j}\rightharpoonup \alpha$ and $|\mu_{n_j}|\rightharpoonup \sigma$ in $\mathrm{M}(\Omega)$ and from \eqref{eq:|mu|leqa}, we have by taking the limit as $j\to\infty$ that
		 		\begin{equation*}
\int_{\Omega} f \, \mathrm{d} |\mu| \le \int_{\Omega} f \, \mathrm{d} \sigma \le \int_{\Omega} f \, \mathrm{d} \alpha. 
\end{equation*}
Finally, since $f\in C_c(\Omega)$ with $f\geq 0$ is arbitrary, $|\mu| \le \alpha$, i.e., $\mu\in \mathbf{K}(\alpha)$ by \Cref{defn_bounded}.
	\end{proof}

The following results show that improving the convergence of $\{\alpha_n\}$ leads to improved convergence for some subsequence $\{\mu_n\}$ such that  $\mu_n\in \mathbf{K}(\alpha_n)$ for all $n\in \mathbb{N}$.

\begin{theorem} \label{thm1}
Suppose that $\{\alpha_n\}$ is a sequence in $\mathrm{M}^+(\Omega)$ such that $\alpha_n\to \alpha$	 in $\mathrm{M}(\Omega)$ for some $\alpha$. Then, every sequence $\{\mu_n\}$ in
\begin{equation*}
	\mathbf{H}=\bigcup_{n=1}^\infty \mathbf{K}(\alpha_n),
\end{equation*}
admits a subsequence that converges in the narrow topology on $\mathrm{M}(\Omega)^N$ to some $\mu\in \mathrm{M}(\Omega)^N$. Further, $\mu$ belongs either to $\mathbf{K}(\alpha_i)$ for some $i\in\mathbb{N}$ or to the narrow closure of $\bigcup_{n=j}^\infty \mathbf{K}(\alpha_n)$ for each $j\in \mathbb{N}$.
\end{theorem}

\begin{proof}
	Given that $\alpha\in \mathrm{M}^+(\Omega)$, $\alpha$ is inner regular (see \cite[Proposition 4.2.1]{abm14}) so that for $\epsilon>0$ there exists a compact set $\Lambda_\epsilon\subset \Omega$ such that 
	\begin{equation*}
		\alpha(\Omega\setminus \Lambda_\epsilon)=\alpha(\Omega)-\alpha(\Lambda_\epsilon)<\frac{\epsilon}{2}.
	\end{equation*}
	
	Since $\alpha_n\to \alpha$	 in $\mathrm{M}(\Omega)$, for the $\epsilon>0$ chosen above there exists an $N_\epsilon\in\mathbb{N}$ such that 
	\begin{equation*}
		|\alpha-\alpha_n|(\Omega)<\frac{\epsilon}{2} \qquad \text{ for } n> N_\epsilon.
	\end{equation*}
	
	Consider $\{\alpha_1,\alpha_2,\ldots,\alpha_{N_\epsilon}\}$, then there exist compact sets $\{\Lambda_\epsilon^1,\Lambda_\epsilon^2,\ldots,\Lambda_\epsilon^{N_\epsilon}\}$ and subsets of $\Omega$ such that  
			\begin{equation*}
		\alpha_n(\Omega\setminus \Lambda^n_\epsilon)<\frac{\epsilon}{2} \qquad n=1,2,\ldots,N_\epsilon.
	\end{equation*}
	Define then
	\begin{equation*}
		\hat{\Lambda}_\epsilon:=\Lambda_\epsilon\cup \Big(\Lambda^1_\epsilon \cup \Lambda^2_\epsilon\cup \cdots \cup \Lambda^{N_\epsilon}_\epsilon\Big),
	\end{equation*}
so that $\hat{\Lambda}_\epsilon\subset \Omega$ is compact and further 
	\begin{equation*}
	\alpha(\Omega\setminus\hat{\Lambda}_\epsilon)<\frac{\epsilon}{2}
\end{equation*}
together with 
\begin{equation*}
	\alpha_n(\Omega\setminus \hat{\Lambda}_\epsilon)<\frac{\epsilon}{2} \qquad n=1,2,\ldots,N_\epsilon.
\end{equation*}
In addition, for $n>N_\epsilon$ we observe
\begin{align*}
	\alpha_n(\Omega\setminus\hat{\Lambda}_\epsilon)&=(\alpha_n-\alpha)(\Omega\setminus\hat{\Lambda}_\epsilon)+\alpha(\Omega\setminus\hat{\Lambda}_\epsilon)\\
	&\leq |\alpha-\alpha_n|(\Omega\setminus\hat{\Lambda}_\epsilon)+\alpha(\Omega\setminus\hat{\Lambda}_\epsilon)\\
	&\leq |\alpha-\alpha_n|(\Omega)+\alpha(\Omega\setminus\hat{\Lambda}_\epsilon)\\
	&<\epsilon,
\end{align*}
so that 
\begin{equation}
	\alpha_n(\Omega\setminus\hat{\Lambda}_\epsilon)<\epsilon,
\end{equation}
for all $n\in \mathbb{N}$.
	
Note that measures in $\mathbf{H}$ are uniformly bounded in the norm of $\mathrm{M}(\Omega)^N$ {  because} $\{\alpha_n(\Omega)\}$ is bounded. Then, if $\mu \in \mathbf{H}$ there exists an $n\in \mathbb{N}$ for which $|\mu|\leq\alpha_n$. Thus 
\begin{equation*}
	\sup\{|\mu|(\Omega\setminus\hat{\Lambda}_\epsilon):\mu\in \mathbf{H}\}\leq\epsilon,
\end{equation*}
and then by Prokohorov's theorem \cite[Theorem 8.6.2. and Theorem 8.6.7.]{bogachev_vol2}
for every sequence $\{\mu_n\}$ in $\mathbf{H}$ there exists a subsequence (not relabelled) such that $\mu_n\xrightarrow[]{\mathrm{nw}}\mu$. In order to prove that $\mu\in \mathbf{H}$, note that if $\{\mu_n\}$ in $\mathbf{H}$ then, $|\mu_n|\leq \alpha_{k(n)}$, where $k:\mathbb{N}\to\mathbb{N}$ is some function. If $k$ is a bounded function, then we can extract a subsequence $\{\alpha_{k(n)}\}$ that is constant and equal to some $\alpha_{k^*}$ for a fixed $k^*\in \mathbb{N}$, and hence extract a subsequence of $\{\mu_n\}$ that satisfies $|\mu_n| \leq \alpha_{k^*}$. Then, by \Cref{lem1} we can extract a further subsequence weakly convergent to $\mu^*$ such that $|\mu^*| \leq \alpha_{k^*}$, but since $\mu_n\xrightarrow[]{\mathrm{nw}}\mu$, then $\mu^*=\mu$ and $\mu \in \mathbf{H}$, and hence $\mu\in \mathbf{K}(\alpha_i)$ for some $i\in\mathbb{N}$.
On the other hand, if $k$ is an unbounded function
then there is some subsequence $\alpha_{k(j_i)} \to
\alpha$ for which $|\mu_{ n _{j_i}} | \le \alpha_{k(j_i)}$ and $\mu\in \overline{\mathbf{H}}^{\:\mathrm{nw}}$ by the same preceding argument and the use of \Cref{lem1}. The same digression can be used to show that $\mu$ belongs to the narrow closure of $\bigcup_{n=j}^\infty \mathbf{K}(\alpha_n)$ for each $j\in \mathbb{N}$.
\end{proof}
	
The previous result leads to the following corollary to be used in the study of perturbations of optimization problems. 

\begin{corollary} \label{corr:narrow}
	Suppose that $\{\alpha_n\}$ is a sequence in $\mathrm{M}^+(\Omega)$ such that $\alpha_n\to \alpha$	 in $\mathrm{M}(\Omega)$ for some $\alpha$. Let $\{\mu_n\}$ be a sequence in $\mathrm{M}(\Omega)^N$ such that $\mu_n\in \mathbf{K}(\alpha_n)$ for $n\in \mathbb{N}$. Then, there exists $\mu\in  \mathbf{K}(\alpha)$ for which $\mu_n\xrightarrow[]{\mathrm{nw}}\mu$ along a subsequence.
\end{corollary}

\begin{proof}
	Since $\alpha_n \to \alpha$  in $\mathrm{M}(\Omega)$
	and $\mu_n \in \mathbf{K}(\alpha_n)$ for each $n \in \mathbb{N}$, 
	it follows from Lemma \ref{lem1} that $\mu_{n} \rightharpoonup \nu$ in $\mathrm{M}(\Omega)^N$ along a subsequence of $\{\mu_{n}\}$ with
	in $\nu\in \mathbf{K}(\alpha)$. Since the sequence $\{ \mu_n \}$ is  in $\cup_{ i =1 }^{\infty} \mathbf{K}(\alpha_n)$, and
	since $ \{ \alpha_n\}$ is in $\mathrm{M}^{+}(\Omega)$ 
	converging strongly to $\alpha \in \mathrm{M}(\Omega)$,
	it follows from Theorem \ref{thm1} that a further 
	subsequence of $\{\mu_{n}\}$ converges narrowly to  some
	$\mu$. Since narrow convergence implies weak convergence, 
	it follows that $\mu = \nu $. 
\end{proof}

{
An analogous corollary holds for both $\mathrm{M}^p(\Omega ; \mathrm{div})$ and $\mathrm{M}_{\Gamma}^p(\Omega ; \mathrm{div})$ provided the sequence $\{\mathrm{div}  \mu_n\}$ is uniformly bounded a priori. 

\begin{corollary} \label{basic_extension_to_div}
	Suppose that $\{\alpha_n\}$ 
	is a sequence in $\mathrm{M}^{+}(\Omega)$ that strongly converges to $\alpha \in \mathrm{M}^{+}(\Omega)$ and that $\{\mu_n\}$ is a sequence with   
	$\mu_n\in  \mathbf{K}^p_{\Gamma}(\alpha_n ; \mathrm{div})$ with $1<p<\infty$. Then, provided that   
	$\sup_{n\in\mathbb{N}}\|\mathrm{div} \mu_n \|_{L^p(\Omega)} < \infty$ holds true, there exists a $\mu^*\in  \mathbf{K}^p_{\Gamma}(\alpha ; \mathrm{div})$ such that $\mu_n \stackrel{\mathrm{nw}}{\longrightarrow} \mu^{*}$
	in $\mathrm{M}(\Omega)^N$ and $\mathrm{div} \mu_n \rightharpoonup \mathrm{div} \mu^{*}$ in $L^p(\Omega)$
	as $ n \to \infty$ (along a subsequence) . 
\end{corollary}

\begin{proof}
	Since $\mu_n$ is in $\mathbf{K}(\alpha_n)$, it follows from \Cref{corr:narrow} that 
	$\mu_n$ narrowly converges to some $\mu^* \in  \mathbf{K}(\alpha)$ along a subsequence. We are left to prove that $\mu^*$ is in $\mathrm{M}^p_\Gamma(\Omega ; \mathrm{div})$.  Since $\|\mathrm{div} \, \mu_n \|_{L^p(\Omega)} < \infty$, there is 
	a further subsequence such that 
		\begin{equation}\label{eq:divmunh}
			 \mathrm{div} \, \mu_n \rightharpoonup h  
		\end{equation}
	in $L^p(\Omega)$ for some $h \in L^p(\Omega)$ as $n \to \infty$. 
	For $\phi \in C_c^{\infty}(\Omega)$ and by
	\Cref{eq:divm} we observe 
		\begin{equation} \label{eq:narrowWk}
			\int_{\Omega} \nabla \phi \cdot { \dif\mu^*}  =
			\lim_{n \to \infty} \int_{\Omega} \nabla \phi \cdot { \dif\mu_{n} } 
			= - \lim_{n \to \infty} \int_{\Omega} \phi \:\mathrm{div} \, \mu_n  \dif x 
			=  - \int_{\Omega} h \:\phi \dif x 
		\end{equation}
	where the left hand side limit is implied by the narrow convergence $\{\mu_n\}$ to $\mu^*$ and the limit on the
	right hand side is due to \eqref{eq:divmunh}; thus $\mathrm{div} \mu^* =h$ and $\mu^*\in \mathrm{M}^p(\Omega ; \mathrm{div})$. 
	
	For $\Gamma\neq\emptyset$, we have that
	$\mu_n\in \mathrm{M}^p_\Gamma(\Omega ; \mathrm{div})$ and hence 
	\begin{equation*}
			\mathrm{N}{ (\phi, \mu_n)}=\int_{\Omega} \nabla \phi \cdot \dif\mu_n+\int_\Omega \phi\: \:\div\:\mu_n \dif x=0
	\end{equation*}  
	for all $\phi \in { \mathbf{C}_b^1}(\overline{\Omega})$ such that $\phi|_{\overline{\Omega\setminus \Gamma }} = 0$. Since $\nabla \phi \in C_{ b}(\Omega)^N$ it follows by the same argument in \eqref{eq:narrowWk} that 
	\begin{equation*}
		\mathrm{N}{ (\phi, \mu^*)}=0,
	\end{equation*}
	{ due to $\mu_n \stackrel{\mathrm{nw}}{\longrightarrow} \mu^{*}$}; thus $\mu^*\in \mathrm{M}^p_\Gamma(\Omega ; \mathrm{div})$.	
\end{proof} }
\begin{remark}
	The previous holds true for the case $p=\infty$ if the weak convergence is replaced by weak-* convergence for $\{\mathrm{div} \mu_n \}$.
\end{remark}

\begin{remark}
It should be noted that the narrow convergence in the conclusion of \Cref{thm1} and 	\Cref{corr:narrow} is the best possible to be expected. Consider for example $\Omega=(0,2\pi)$, $\alpha$ be the Lebesgue measure, and let $\mu_n=\sin(n x)\:\alpha$ so that $|\mu_n|\leq \alpha $. Further, $\mu_n\rightharpoonup 0$ and $\mu_n\xrightarrow[]{\mathrm{nw}} 0$, however $\mu_n$ does not converge to zero strongly, as $|\mu_n|(\Omega)=4$.
\end{remark}

\subsection{Backward results}\label{sec:M2}

{
The previous \Cref{corr:narrow} shows that a sequence of measures $\mu_n \in \mathbf{K}(\alpha_n)$ 
converges (along a subsequence) narrowly to a measure $\mu \in \mathbf{K}(\alpha)$ provided that $\alpha$ is the 
strong limit of the sequence $\{\alpha_n\}$. In fact, the following converse result can be obtained under the same assumptions: For a given $\mu \in \mathbf{K}(\alpha)$ we can find a ``recovery'' sequence $\mu_n \in \mathbf{K}(\alpha_n)$
that converges in norm to $\mu$. We show this in \Cref{thm:Recovery1},
which follows after the next classical lemma for the total variation of mutually singular measures.} 

Recall the following standard definitions: Given two measures $\mu \in \mathrm{M}(\Omega)^N$, and $\alpha \in \mathrm{M}^{+}(\Omega)$, we say that $\mu$ is \emph{absolutely continuous with respect to the measure} $\alpha$, and we denoted it as $\mu\ll \alpha$, if for every Borel set $B$ such that $\alpha(B)=0$ then $\mu(B)=0$. Further, we say that the measure $\mu$ is \emph{singular with respect to} $\alpha$, denoted as $\mu \perp\alpha$, if there exists a Borel set $B$ such that $\alpha(B)=0$ and $\mu$ is concentrated on $B$, i.e., $\mu(C)=0$ for all Borel sets such that $B\cap C=\emptyset$. The support of $\mu$, denoted as $\mathrm{supp}\,\mu$, is the smallest closed set $C\subset \Omega$ such that $|\mu|(\Omega\setminus C)=0$ and it can be proven that equivalently:
\begin{equation*}
	\mathrm{supp}\,\mu=\{x\in \Omega: \forall r>0, |\mu|(B_r(x))>0\},
\end{equation*}
where $B_r(x)=\{y\in \Omega: |x-y|<r\}$.

The set of (equivalence classes of) functions $f:\Omega\to \mathbb{R}^{N}$ such that
\begin{equation*}
	\int_\Omega |f|\dif \alpha <+\infty,
\end{equation*}
is denoted as $L^1(\Omega,\alpha)^N$. If for $f=\{f_i\}_{i=1}^N$ and $\mu=\{\mu_i\}_{i=1}^N$, we have that $f_i\in L^1(\Omega,\mu_i)$ for $i=1,\ldots,N$, then we write that $f\in L^1(\Omega,\mu)$.

We start with the result of Lebesgue and Radon-Nikodym decompositions (see \cite[Theorem 4.2.1]{abm14} and \cite{rudin1987real}) in our { vector} setting.

\begin{lemma} \label{ezlemma}
	Let $\mu \in \mathrm{M}(\Omega)^N$ and $\alpha \in \mathrm{M}^{+}(\Omega)$. Then, there exists $F\in L^{1}(\Omega , \alpha)^N$ and $\mu^{s} \in \mathrm{M}(\Omega)^N$
such that 
\begin{equation*}
	\mu(B) = \int_{B} F \, \mathrm{d} \alpha + \mu^{s}(B),
\end{equation*}
for each Borel set $B\subset \Omega$  with $\mu^{s} \perp \alpha$, and for which
\begin{equation*}
	|\mu|(B) = \int_{B} |F| \, \mathrm{d} \alpha + |\mu^{s}|(B).
\end{equation*}
\end{lemma}

\begin{proof} The fact that $|\mu| = |F\alpha| + |\mu^s|$ is a corollary 
to the Lebesgue decomposition, which exists by assumption \cite[Theorem 4.2.1]{abm14}.
Finally, since $F\in L^{1}(\Omega , \alpha)^N$ and $\alpha$ is positive, a standard result gives
$|F\alpha| = |F|\alpha$ and proves the claim \cite[Proposition 1.23]{afp00}. 
\end{proof}

The function $F$ above is commonly written as $\frac{\dif \mu}{\dif  \alpha}$ and called the Radon-Nikodym derivative, and it is unique up to a set of $\alpha$-measure zero. The $\alpha$-integrability of $F$ is a consequences of the fact that $|\mu|(\Omega)<+\infty$. 

The following result represents the initial construction of the recovery sequence in the case $\mathrm{M}(\Omega)^N$ and associated to $\alpha\mapsto \mathbf{K}(\alpha)$. The construction of the recovery sequence is done by means of scaling via the Radon-Nikodym derivative $\frac{\dif \alpha^a_n}{\dif \alpha}$ (where $\alpha^a_n$ is the absolutely continuous part of the Lebesgue decomposition with respect to $\alpha$) as we see next.

\begin{theorem}\label{thm:Recovery1}
Suppose that $\{\alpha_n\}$ is a sequence in $\mathrm{M}^+(\Omega)$ such that $\alpha_n\to \alpha$	 in $\mathrm{M}(\Omega)$ for some $\alpha$, and that $\mu\in \mathbf{K}(\alpha)$ is arbitrary. Then, there exists a sequence $\{\mu_n\}$ in $\mathrm{M}(\Omega)^N$ such that $\mu_n\in \mathbf{K}(\alpha_n)$ for $n\in\mathbb{N}$ and $\mu_n\to \mu$ in $\mathrm{M}(\Omega)^N$.
	\end{theorem}

	\begin{proof}
		Any $\mu\in \mathbf{K}(\alpha)$ satisfies $|\mu|\leq \alpha$ so that $\mu\ll \alpha$. Then, by the Radon-Nikodym decomposition  we have $\mu=F \alpha$ or equivalently
	\begin{equation*}
		\mu( B ) = \int_{B} F \, \mathrm{d}\alpha,
	\end{equation*}
	for any Borel set $B$ and
	where $F:\Omega\to \mathbb{R}^N$ is such that $F\in L^1(\Omega, \alpha)^N$. Given that
	$|\mu|=|F|\alpha$ by \cite[Proposition 1.23]{afp00} and $|\mu| \le \alpha$,
	it follows that $|F| \le 1$ .
	
	Since $\alpha_n\in \mathrm{M}^+(\Omega)$ for all $n\in \mathbb{N}$, then again by the Lebesgue and  the Radon-Nikodym Decomposition Theorem we observe that
	\begin{equation*}
		\alpha_n=g_n\alpha+\alpha_n^s,
	\end{equation*}
	where $g_n:\Omega\to \mathbb{R}$ is such that 
	$g_n\in L^1(\Omega;\alpha)^+$ and $\alpha_n^s \perp \alpha$ with $\alpha_n^s\in\mathrm{M}^+(\Omega)$. 
	Since $\alpha(\Omega) = \int_{\Omega} 1 \, \mathrm{d}\alpha$,
	we also have 
	\begin{equation*}
		(\alpha_n - \alpha)( \Omega )  
		= \int_{\Omega} g_n - 1 \, \mathrm{d}\alpha +
		\alpha_n^s(\Omega).
	\end{equation*}
	It follows from Lemma \ref{ezlemma} that 
		\begin{equation}
			\| \alpha_n - \alpha\|_{\mathrm{M}(\Omega)} = 
			\int_{\Omega} |g_n - 1| \, \mathrm{d}\alpha +
		\alpha_n^s(\Omega) \geq \int_{\Omega} |g_n - 1| \, \mathrm{d}\alpha, 
		\end{equation}
	given that $\alpha_n^s\geq 0$. Since $| \alpha_n - \alpha |(\Omega) \to 0$ by assumption, then
	$\| g_n - 1 \|_{L^1(\Omega , \alpha)} \to 0 $ as well.
	Define the sequence  $\{F_n\}$ in $L^1(\Omega; \alpha)^N$ as 
	\begin{equation*}
		F_n=g_nF,
	\end{equation*}
	Further, define for each $n\in \mathbb{N}$ the measure $\mu_n\in \mathrm{M}(\Omega)^N$ as $\mu_n=F_n\alpha$, that is for every Borel set $B$ we have		\begin{equation*}
		\mu_n(B)=\int_{ B}F_n\dif \alpha.
	\end{equation*} 
	Note that since $g_n\geq 0$ and $|F|\leq 1$ then
	\begin{equation*}
		|\mu_n|{ =} |F_n|\alpha { =} |F| g_n\alpha\leq g_n \alpha\leq\alpha_n,
	\end{equation*}
	that is $\mu_n\in \mathbf{K}(\alpha_n)$. Finally, since $|F| \le 1 $, we obtain that
	\begin{align*}
		\limsup_{n\to\infty}\| \mu_n - \mu \|_{\mathrm{M}(\Omega )^N} &\le 
		\limsup_{n\to\infty}\int_{\Omega} | F_n - F | \, \mathrm{d} \alpha \\
		&=\limsup_{n\to\infty} \int_{\Omega} |F|| g_n - 1| \, \mathrm{d} \alpha \\
		&\le \limsup_{n\to\infty}\| g_n - 1 \|_{L^1(\Omega , \alpha)}=0,
	\end{align*}
	which shows that $\mu_n \to \mu$  in norm and concludes the result. 
	\end{proof}

	For the sake of simplicity, we consider the following notation. Let $\mu\in\mathrm{M}(\Omega)^N$,  $\sigma\in \mathrm{M}^+(\Omega)$ and let $\mu=\mu^a+\mu^s$ be the associated Lebesque decomposition where $\mu^a\ll \sigma$ and $\mu^s\perp \sigma$. We denote by $F^\mu_\sigma:\Omega\to\mathbb{R}^N$ the function in $L_{\sigma}^1(\Omega)^N$ such that 
	\begin{equation*}
		\mu^a(B)=\int_B F^\mu_\sigma \dif \sigma,
	\end{equation*}
	for any Borel set $B$ in $\Omega$. The existence of $F^\mu_\sigma$ is guaranteed by the Radon-Nikodym Theorem. If $\sigma=\mathcal{L}^N$, the $N$-dimensional Lebesgue measure, we omit the subscript ``$\mathcal{L}^N$'' and  write $F^\mu:=F^\mu_{\mathcal{L}^N}$.
	
	In the following lemma, we show that if $\mu\in \mathrm{DM}(\Omega)$ then the measure defined by $\nu=g \mu$ where $g$ is $\mu$-integrable, smooth and with $\mu$-integrable gradient is also in $ \mathrm{DM}(\Omega)$. If additionally, $\mu\in \mathrm{M}^p(\Omega;\div )$, then in order to conclude that $\nu\in \mathrm{M}^p(\Omega;\div )$ additional structural assumptions are required not only on $g$ but also on $\mu$ as we see next.

	\begin{lemma}\label{lem:DM}
		Let $g:\Omega\to \mathbb{R}$ be bounded, $g\in C^1(\Omega)$, and also $\nabla g\in L^1(\Omega, \mu)$ for some $\mu\in \mathrm{M}(\Omega)^{N}$. Define the set function $\nu$ as
		\begin{equation*}
			\nu(B)=\int_{ B} g \dif \mu,
		\end{equation*}
		for any Borel set $B\subset \Omega$. Hence,
		\begin{itemize}
			\item[$\mathrm{(i)}$] if $\mu\in \mathrm{DM}(\Omega)$, then $\nu \in \mathrm{DM}(\Omega)$ and its divergence $\div\: \nu$ is given by
		\begin{equation*}
			\div\: \nu (B)= \int_{ B} g \:\dif \div\,\mu + \int_{B} \nabla g \cdot\dif \mu,
		\end{equation*}
		for any Borel set $B\subset \Omega$.
		\item[$\mathrm{(ii)}$] In addition, suppose $\mu\in \mathrm{M}^p(\Omega;\div )$ for $1\leq p\leq +\infty$,  and $\nabla g\cdot F^\mu\in L^p(\Omega)$. Then, $\nu\in \mathrm{M}^p(\Omega;\div )$ provided that $\nabla g$ vanishes in the support of the measure that is singular to the Lebesgue measure in the Lebesgue decomposition of $\mu$, that is, $\nabla g=0$ in  $\mathrm{supp}\: \mu^s$ where $\mu=F^\mu \mathcal{L}^N+\mu^s$ is the Lebesgue decomposition of $\mu$ with respect to the $N$-dimensional Lebesgue measure $\mathcal{L}^N$. The divergence of $\nu$ in this case is given by
		\begin{equation*}
			\div \,\nu =g \,\div\, \mu+\nabla g\cdot F^\mu.
		\end{equation*}
		\end{itemize}
{  Furthermore, if $\partial\Omega$ is not empty, and} we assume that $g\in { \mathbf{C}_b^1}(\overline{\Omega})$, then $\mathrm{(i)}$ and $\mathrm{(ii)}$ hold true exchanging $\mathrm{DM}(\Omega)$ by $\mathrm{DM}_\Gamma(\Omega)$, and $\mathrm{M}^p(\Omega;\div)$  by $\mathrm{M}^p_\Gamma(\Omega;\div)$, {  for a non-empty $\Gamma\subset \partial\Omega$}.
		
			\end{lemma}
	
	\begin{proof} \label{repr}
	
	Note initially that since $g$ is continuous and bounded, we have that $\nu$ is in $\mathrm{M}(\Omega)^N$. Next, concerning the definition of $\mathrm{DM}(\Omega)$,  note that the test function $\phi$ in \eqref{eq:divm} within \Cref{def:DM}   can be exchanged from $C_c^\infty(\Omega)$ to $C_c^1(\Omega)$. Let $\mu\in \mathrm{DM}(\Omega)$ and let $\phi\in C_c^1(\Omega)$ be arbitrary. Further, let $\{\phi_n\}$ be a sequence in $C_c^\infty(\Omega)$  {  for which there exists a compact set $K$ for which $\mathrm{supp} (\phi_n)\subset K$ for all $n$, and} such that $\phi_n\to\phi$ { and $\nabla\phi_n\to\nabla\phi$ converge uniformly}; the existence of $\{\phi_n\}$ follows by  standard mollification techniques. Since $\mu\in\mathrm{DM}(\Omega)$ then
				\begin{equation*}
		\int_{\Omega} \nabla \phi_n \cdot \mathrm{d} \mu 
		= - \int_{\Omega} \phi_n \: \dif \div \mu.
				\end{equation*}
			{  Also, given that } $\nabla\phi_n\to\nabla\phi$ and $\phi_n\to\phi$ {  uniformly}, and $\phi_n,\nabla\phi_n\in C_c(\Omega)$ by taking the limit above we obtain
							\begin{equation*}
		\int_{\Omega} \nabla \phi \cdot \mathrm{d} \mu 
		= - \int_{\Omega} \phi \: \dif \div \mu.
				\end{equation*} 	
	{ Finally, } since 	$\phi\in C_c^1(\Omega)$ was arbitrary, we can consider test functions in this space.	
	
	Let $\phi \in C_c^1(\Omega)$ be arbitrary, then since $g\in C^1(\Omega)$, we observe that $g\nabla \phi =\nabla(g\phi)-\phi\nabla g$. Further, since $\nabla g\in L^1(\Omega,\mu)$, then $\phi \nabla g\in L^1(\Omega,\mu)$ and hence
	\begin{align*}
		\int_{\Omega} \nabla \phi \cdot \mathrm{d} \nu &=\int_{\Omega} g\nabla \phi \cdot \mathrm{d} \mu=\int_{\Omega} \nabla (g\phi) \cdot \mathrm{d} \mu-\int_{\Omega} \phi\nabla g \cdot \mathrm{d} \mu.
	\end{align*}
	Given that $g\phi \in C_c^1(\Omega)$ and $\mu\in\mathrm{DM}(\Omega)$, we have
	\begin{align*}
		\int_{\Omega} \nabla \phi \cdot \mathrm{d} \nu =-\int_{\Omega} \phi g \,\mathrm{d} \div \mu-\int_{\Omega} \phi\nabla g \cdot \mathrm{d} \mu,
	\end{align*}
		which proves (i).
		
		In order to prove (ii), consider  the Lebesgue decomposition of $\mu$ with respect to the $N$-dimensional Lebesgue measure $\mathcal{L}^N$, i.e., $\mu=F^\mu \mathcal{L}^N+\mu^s$. Since $\nabla g=0$ in $\mathrm{supp}\,\mu^s$ then
		\begin{equation*}
			\int_{\Omega} \phi\nabla g \cdot \mathrm{d} \mu=\int_{\Omega} \phi\nabla g \cdot F^\mu \mathrm{d} x.
		\end{equation*}
		Hence, if $\mu\in \mathrm{M}^p(\Omega;\div )$, then
			\begin{align*}
		\int_{\Omega} \nabla \phi \cdot \mathrm{d} \nu =-\int_{\Omega} \phi (g \, \div \mu +\nabla g \cdot F^\mu) \mathrm{d} x.
	\end{align*}
	If in addition, $g$ is bounded and $\nabla g \cdot F^\mu\in L^p(\Omega)$, then $\nu\in \mathrm{M}^p(\Omega;\div )$, and (ii) is proven.
	
	Let $g\in { \mathbf{C}_b^1}(\overline{\Omega})$ and $\phi\in { \mathbf{C}_b^1}(\overline{\Omega})$ { be} such that $\phi(x)=0$ for all $x\in  \overline{\partial \Omega \setminus \Gamma}$. It follows that $g\phi$ also vanishes at $\overline{\partial \Omega \setminus \Gamma}$. If $\mu\in \mathrm{DM}_\Gamma(\Omega)$, then
	\begin{equation*}
	\int_{\Omega} \nabla (g\phi) \cdot \mathrm{d} \mu=-\int_{\Omega} \phi g \,\mathrm{d} \div \mu,	
	\end{equation*}
	so that
	\begin{equation*}
		\mathrm{N}{ (\phi, \nu)}=\int_{\Omega} \phi \: \mathrm{d}  \div\nu+\int_{\Omega} \nabla \phi \cdot \mathrm{d} \nu=0.
	\end{equation*}
	Then, $\nu\in \mathrm{DM}_\Gamma(\Omega)$ given that $\phi\in { \mathbf{C}_b^1}(\overline{\Omega})$ with $\phi(x)=0$ for all $x\in  \overline{\partial \Omega \setminus \Gamma}$ was arbitrary. Further, if $\mu\in \mathrm{M}^p_\Gamma(\Omega;\div )$ and the assumptions of (ii) hold true, then $\nu\in \mathrm{M}^p_\Gamma(\Omega;\div )$. 

	\end{proof}

\begin{remark}
	It should be noted that $g\in { \mathbf{C}_b^1}(\overline{\Omega})$ and $\nabla g=0$ in  $\mathrm{supp} \mu^s$ where $\mu=F^\mu \mathcal{L}^N+\mu^s$ is sufficient for all the assumptions concerning $g$ in the previous theorem to hold true.
\end{remark}

{ The technical lemma that we introduced above 
allows us to prove the existence of ``recovery sequences''
for both $ \mathrm{DM}_\Gamma(\Omega)$ and $\mathrm{M}_\Gamma^p(\Omega ; \mathrm{div})$. Specifically, for a sequence $\{\alpha_n\}$, and $\alpha$  in $\mathrm{M}^+(\Omega)$, the Lebesgue decomposition (with respect to $\alpha$) leads to $\alpha_n= \alpha_n^a+\alpha_n^s$ and hence (almost) all conditions can be determined by regularity and convergence properties of the Radon-Nikodym $\{\frac{\dif \alpha_n^a}{\dif \alpha}\}$ as we show next in the main result of the paper. 
 }
	
	\begin{theorem} \label{recovery_sequence}
Suppose that $\{\alpha_n\}$ is a sequence in $\mathrm{M}^+(\Omega)$ with $\alpha\in \mathrm{M}^+(\Omega)$ as well. Assume that for the Lebesgue decomposition for $\alpha_n$ with respect to $\alpha$, given by
\begin{equation*}
	\alpha_n=g_n\alpha+\alpha_n^s,
\end{equation*}
we observe that $\alpha_n^s\to 0$ in $\mathrm{M}(\Omega)$ as $n\to\infty$. Further, for all $n\in\mathbb{N}$, $g_n \in C^1(\Omega)$, $g_n$ is bounded, $\nabla g_n\in L^1(\Omega,\alpha)$, and  
\begin{equation*}
	 \sup_{x\in\Omega}|g_n-1|\to 0,
\end{equation*}
as $n\to \infty$.
\begin{itemize}
	\item[$\mathrm{(i)}$] If 
	\begin{equation*}
	  \int_\Omega|\nabla g_n| \dif\alpha \to 0,
\end{equation*}
as $n\to\infty$, then, for $\mu\in K_0(\alpha)$ arbitrary, where 
\begin{equation*}
K_0(\alpha)=	\{\sigma \in \mathrm{DM}(\Omega) : |\sigma| \leq \alpha \},
\end{equation*}
there exists a sequence $\{\mu_n\}$ in $\mathrm{DM}(\Omega)$ such that $\mu_n\in K_0(\alpha_n)$ for $n\in\mathbb{N}$ and $\mu_n\to \mu$ in $\mathrm{DM}(\Omega)$ as $n\to\infty$. 

	\item[$\mathrm{(ii)}$] Suppose that, for each $n\in\mathbb{N}$, $\nabla g_n$ vanishes on $\mathrm{supp} \,\alpha^s$, the support of the measure $\alpha^s$ where 
	\begin{equation*}
		\alpha=F^ \alpha \mathcal{L}^N+\alpha^s
	\end{equation*}
	is the Lebesgue and Radon-Nikodym decomposition of $\alpha$. Let $\mu\in \mathbf{K}^p(\alpha; \mathrm{div})$ be arbitrary where 
\begin{equation*}
\mathbf{K}^p(\alpha; \mathrm{div})=	\{\sigma \in \mathrm{M}^p(\Omega;\div ) : |\sigma| \leq \alpha \},
\end{equation*}
for $1\leq p\leq +\infty$, and suppose that 
\begin{equation*}
	\||\nabla g_n| F^\alpha\|_{L^p(\Omega)} \to 0.
\end{equation*}
Then, there exists a sequence $\{\mu_n\}$ in $\mathrm{M}^p(\Omega;\div )$ such that $\mu_n\in \mathbf{K}^p(\alpha_n; \mathrm{div})$ for $n\in\mathbb{N}$ and $\mu_n\to \mu$ in $\mathrm{M}^p(\Omega;\div )$ as $n\to\infty$. 

\item[$(\mathrm{iii})$] If { $\partial\Omega$ is not empty}, and in $\mathrm{(i)}$ and $\mathrm{(ii)}$ we assume in addition that $g_n\in { \mathbf{C}_b^1}(\overline{\Omega})$, then their respective results hold true when exchanging $\mathrm{DM}(\Omega)$ by $\mathrm{DM}_\Gamma(\Omega)$ in the definition of $K_0$, and $\mathrm{M}^p(\Omega;\div)$  by $\mathrm{M}^p_\Gamma(\Omega;\div)$, i.e.,  by exchanging $\mathbf{K}^p(\alpha_n; \mathrm{div})$ by $\mathbf{K}_\Gamma^p(\alpha_n; \mathrm{div})$.
 
\end{itemize}

	\end{theorem}

	\begin{proof}
	Given that the Lebesgue and Radon-Nikodym decomposition for $\alpha_n$ with respect to $\alpha$, and determined by
$\alpha_n=g_n\alpha+\alpha_n^s,$ satisfies $\alpha_n^s\to 0$ in $\mathrm{M}(\Omega)$, and $\sup_{x\in\Omega}|g_n(x)-1|\to 0$ as $n\to\infty$, we initially observe that
	\begin{equation*}
		\alpha_n\to\alpha \quad \text{in } \mathrm{M}(\Omega),
	\end{equation*}
	as $n\to\infty$. Hence, the conclusion of \Cref{thm:Recovery1} holds true; in particular by the construction of its proof: For arbitrary $\mu\in K_0(\alpha)$ we define
		\begin{equation*}
			\mu_n(B)=\int_{B}g_n\dif\mu,
		\end{equation*}
		where $B$ is a Borel subset of $\Omega$. For each $n\in\mathbb{N}$, $\mu_n$ is well-defined given that $g_n$ is continuous and bounded, and $\mu\ll \alpha$ since $|\mu|\leq \alpha$. It follows that $|\mu_n|\leq \alpha_n$, i.e., $\mu_n\in K_0(\alpha_n)$ for $n\in\mathbb{N}$, and $\mu_n\to \mu$ in $\mathrm{M}(\Omega)^N$ as $n\to\infty$.
		
		Further, since $g_n$ is in $C^1(\Omega)$, it is bounded, and $\nabla g_n\in L^1(\Omega,\alpha)$, and hence in $\nabla g_n\in L^1(\Omega,\mu)$, by \Cref{lem:DM} we have that $\mu_n\in \mathrm{DM}(\Omega)$ and also
				\begin{equation*}
			\div\: \mu_n (B)= \int_{  B} g_n \:\dif \div\,\mu + \int_{  B} \nabla g_n \cdot\dif \mu.
		\end{equation*}
		Thus for an arbitrary $\varphi\in C_c(\Omega)$ with $|\varphi|\leq 1$, we have
		\begin{align*}
			|\langle\div\: \mu_n-\div\: \mu, \varphi\rangle_{\mathrm{M}(\Omega),C_c(\Omega)}|&=\left|\int_{\Omega} \varphi(g_n-1) \:\dif \div\,\mu + \int_{\Omega} \varphi \nabla g_n \cdot\dif \mu\right|\\
			&\leq |\div\,\mu(\Omega)| \left(\sup_{x\in\Omega}|g_n(x)-1|\right)+\int_\Omega|\nabla g_n| \dif\alpha,
		\end{align*}
		where we have used that $|\mu|\leq\alpha$. By taking the supremum over all $\varphi\in C_c(\Omega)$ with $|\varphi|\leq 1$, and subsequently taking the limit as $n\to\infty$ we observe that
		\begin{equation*}
			\lim_{n\to\infty} |\div\: \mu_n-\div\: \mu|(\Omega)=0,
		\end{equation*}
		or equivalently $\div\: \mu_n\to\div\: \mu$ in $\mathrm{M}(\Omega)$, and hence $\mu_n\to \mu$ in $\mathrm{DM}(\Omega)$ as $n\to\infty$.

		We focus on (ii) next. Since $\nabla g_n=0$ in $\mathrm{supp} \,\alpha^s$, the support of the measure $\alpha^s$, then we claim that 
		\begin{equation*}
			 \int_{  B} \nabla g_n \cdot\dif \mu =  \int_{  B} \nabla g_n \cdot F^\mu\dif x,
		\end{equation*}
		{  for all Borel sets $B$},
		where $\mu=F^\mu \mathcal{L}^N+\mu^s$. Since $\mu\in \mathbf{K}^p(\alpha; \mathrm{div})$, then $|\mu|\leq \alpha$ which implies that
		\begin{equation*}
			|F^\mu|\leq F^\alpha  \qquad \text{ and }\qquad|\mu^s|\leq \alpha^s,
		\end{equation*}
		where the first inequality holds pointwise a.e. with respect to the Lebesgue measure and the second one in  measure sense. In particular, the latter implies that $\mu^s\ll \alpha^s$ so that $\nabla g_n=0$ in $\mathrm{supp} \,\mu^s$ as well. Hence, $\int_{  B} \nabla g_n \cdot\dif \mu^s=0$ which proves the claim. Further,
		
		\begin{align*}
			\|\div\: \mu_n-\div\: \mu\|_{L^p(\Omega)}&\leq\||(g_n-1) \: \div\,\mu\|_{L^p(\Omega)}+\|\nabla g_n \cdot F^\mu\|_{L^p(\Omega)}\\
			&\leq \left(\sup_{x\in\Omega}|g_n(x)-1|\right)\|\div\,\mu\|_{L^p(\Omega)}+\||\nabla g_n| F^\alpha\|_{L^p(\Omega)},
		\end{align*}
		where we have used that $|\nabla g_n \cdot F^\mu|\leq |\nabla g_n | F^\alpha$. Therefore,
		
				\begin{equation*}
			\lim_{n\to\infty} \|\div\: \mu_n-\div\: \mu\|_{L^p(\Omega)}=0,
		\end{equation*}
		and thus  $\mu_n\to \mu$ in $\mathrm{M}^p(\Omega;\div )$ as $n\to\infty$ and the result is proven.
		
		Finally, we consider on (iii).  Since \Cref{lem:DM} holds for $\mathrm{DM}_{\Gamma}(\Omega)$ 
		and $\mathrm{M}_{\Gamma}(\Omega ; \mathrm{div})$ provided 
		$g_n \in { \mathbf{C}_b^1}(\overline{\Omega})$, conditions (i) and (ii) of 
		\Cref{recovery_sequence} also hold for $g_n \in { \mathbf{C}_b^1}(\overline{\Omega})$.
		
	\end{proof}

\begin{remark}
	It should be noted that sufficient conditions for all instances in the theorem above for the sequence  of functions $\{g_n\}$ are that $g_n\in { \mathbf{C}_b^1}(\overline{\Omega})$ for $n\in \mathbb{N}$, $g_n$ is constant on a neighborhood of $\mathrm{supp} \,\alpha^s$,  and that 
	\begin{equation*}
			 \sup_{x\in\Omega}|g_n(x)-1|\to 0, \qquad \text{and} \qquad \sup_{x\in\Omega}|\nabla g_n(x)|\to 0,
	\end{equation*}
	both as $n\to\infty$.
\end{remark}
	
	\section{Application to optimization problems }

In this section we apply the results of the previous one to optimization problems that arise in applications as described in \Cref{sec:formal}. Consider the following optimization problem over the space $\mathrm{M}_{\Gamma}^p(\Omega ; \mathrm{div})$ with total variation constraints:
	 
		\begin{equation} \label{eq:P} \tag{$\mathbb{P}$}
			\begin{aligned}
			\min_{\mu} \quad & \mathcal{J}(\mu) := \frac{1}{p} \int_{\Omega} \left| \mathrm{div}\, \mu(x) - f(x)\right|^p \, \mathrm{d} x   
			+ \int_{\Omega} \beta(x) \, \mathrm{d}|\mu|(x)\\
			\textrm{s.t.} \quad & \mu \in \mathrm{M}_{\Gamma}^p(\Omega ; \mathrm{div})\\
			&|\mu| \le \alpha 
			\end{aligned}
		\end{equation}
	for $\alpha \in \mathrm{M}^{+}(\Omega)$, $f \in L^p(\Omega)$, $1 < p <  \infty$, and 
	and $\beta$ a non-negative continuous and bounded function. In cases where we need to study the dependence of the problem with respect to $\alpha$, 
	we use the notation $\mathbb{P}(\alpha)$. Further note that the problem can be written as $\min \mathcal{J}(\mu)$  subject to $\mu\in \mathbf{K}^p_{\Gamma}(\alpha; \mathrm{div})$.
	\begin{theorem} \label{optimize_static}
		The problem \eqref{eq:P} admits solutions.
	\end{theorem}
	\begin{proof} 
		We follow the direct method. The functional $\mathcal{J}$ is bounded 
		from below and $\mathbf{K}^p_{\Gamma}( \alpha ; \mathrm{div})$ is non-empty, so choose an infimizing sequence $\{ \mu_n \} _ { n = 1 } ^ { \infty }$
		with $\mu_n \in \mathbf{K}^p_{\Gamma}( \alpha ; \mathrm{div})$ for $n\in\mathbb{N}$ such that 
			\begin{equation*}
				\lim_{n\to\infty}\mathcal{J}( \mu_n ) \to M = \inf \mathcal{J}(\mu) \quad \text{s.t.} \:\:\mu \in  \mathbf{K}^p_{\Gamma}( \alpha ; \mathrm{div}).
			\end{equation*} 
		Since $\{\| \mathrm{div} \mu_n \|_{L^p(\Omega)}\}_{n=1}^{\infty}$ is bounded due 
		to the structure of $\mathcal{J}$ and $|\mu_n| \le \alpha $ 
		for every $n$, it follows from \Cref{basic_extension_to_div} that 
		there is some $\mu^* \in \mathbf{K}^p_{\Gamma}(\alpha ; \mathrm{div})$ 
		for which $\mu_{n}\xrightarrow[]{\textrm{nw}}\mu^*$ 
		and $\mathrm{div} \mu_n \rightharpoonup \mathrm{div} \mu^*$ in $L^p(\Omega)$
		along a subsequence (not relabelled) as $n \to \infty$.  We claim that
                
			\begin{equation*}
				\mathcal{J}(\mu^*) \le \liminf_{n} \mathcal{J}(\mu_{n} ) 
			\quad \text{for} \quad \mu_{n} \rightharpoonup \mu^* .
			\end{equation*}
		Since $\mu_{n}\xrightarrow[]{\textrm{nw}}\mu^*$ then we also have that $\nu_{n}\xrightarrow[]{\textrm{nw}}\nu^*$ where $\nu_n:=\beta\mu_n$ with $n\in \mathbb{N}$ and $\nu^*=\beta\mu^*$, that is
		\begin{equation*}
			\nu_n(B)=\int_B\beta\dif\mu_n, \quad\text{for}\quad n\in\mathbb{N} \qquad \text{and} \qquad \nu^*(B)=\int_B\beta\dif\mu^*,
		\end{equation*}
		for any Borel set $B\subset \Omega$. Since $|\nu_n|=\beta |\mu_n|$ it follows by { 
		Corollary 4.2.1. in \cite{abm14}} that 
			\begin{equation*}
				\int_{\Omega} \beta \, \mathrm{d}|\mu^*| 
				\le \liminf_{n}  \int_{\Omega} \beta \, \mathrm{d}|\mu_{n}|.			\end{equation*} 
		Next, observe that the functional 
			\begin{equation*}
				L^p(\Omega)\ni v\mapsto \frac{1}{p} \int_{\Omega} \left| v - f \right|^p \, \mathrm{d} x  
			\end{equation*}
		is weakly lower semicontinuous given that it is both continuous and 
		convex. Therefore,
		\begin{equation*}
			M \le \mathcal{J}(\mu^*) \le \liminf_{n} \mathcal{J}(\mu_{n} ) 
		\end{equation*}
		Since $\liminf_n \mathcal{J}(\mu_n ) = M$, 
		it follows that $\mu^*$ minimizes $\mathcal{J}$. 
	\end{proof}

Now we are in a position to address a stability result associated with solutions to  $\mathbb{P}(\alpha)$ with respect to perturbations of $\alpha$. In particular, the result hinges on both forward and backward results associated to the convergence of $\alpha\mapsto \mathbf{K}^p_{\Gamma}(\alpha ; \mathrm{div})$.

	\begin{theorem}\label{optimize_dynamic}
		Let $\{\alpha_n\}$ be a sequence measures in $\mathrm{M}^{+}(\Omega)$ that converges to $\alpha \in \mathrm{M}^{+}(\Omega)$
		in norm and satisfies the conditions of $(\mathrm{iii})$ in \Cref{recovery_sequence}.
		For each $\alpha_n$, a solution $\mu_n$ to the problem 
		$\mathbb{P}(\alpha_n)$ exists for which
		\begin{equation*}
			\mu_n \stackrel{\mathrm{nw}}{\longrightarrow}
		\mu^* \qquad\text{and}\qquad  \mathrm{div}\, \mu_n \rightharpoonup \mathrm{div} \, \mu^* \:\text{ in }\: L^p(\Omega),
		\end{equation*}
		along a subsequence (not relabeled) as $n\to\infty$ for some $\mu^* \in \mathbf{K}^p_{\Gamma}(\alpha ; \mathrm{div})$ that solves $\mathbb{P}(\alpha)$.

	\end{theorem}

	\begin{proof}
		By \Cref{optimize_static} each problem $\mathbb{P}(\alpha_n)$
		has a solution $\mu_n \in \mathbf{K}^p_{\Gamma}(\alpha_n ; \mathrm{div})$. It then follows from  \Cref{basic_extension_to_div} that $\mu_n \stackrel{\mathrm{nw}}{\longrightarrow}
		\mu^*$ and $\mathrm{div}\, \mu_n \rightharpoonup \mathrm{div} \, \mu^*$ in $L^p(\Omega)$ along a subsequence
		for some measure $\mu^* \in \mathbf{K}^p_{\Gamma}(\alpha ; \mathrm{div})$. 
		
		We now show  that $\mu^*$ solves $\mathbb{P}(\alpha)$. Let $\nu \in \mathbf{K}^p_{\Gamma}(\alpha ; \mathrm{div})$ be arbitrary. 
		Since we assumed that the sequence $\{\alpha_n\}_{n =1 }^{\infty}$ satisfies the assumptions required to apply \Cref{recovery_sequence}, there exists 
		a sequence $\{\nu_n\}_{n=1}^{\infty}$ with $\nu_n \in \mathbf{K}^p_{\Gamma}(\alpha ; \mathrm{div})$ such that $\nu_n \to \nu $ in $\mathrm{M}^p_\Gamma(\Omega;\div)$ as $n \to \infty$.
		Exploiting that $\mu_n$ is a minimizer to $\mathbb{P}(\alpha_n)$, we observe
		\begin{equation*}
			\mathcal{J}(\mu_n) \le \mathcal{J}(\nu_n)
		\end{equation*}
		for all indices $n$. It then follows from lower semicontinuity of $\mathcal{J}$ for $\mu_n \stackrel{\mathrm{nw}}{\longrightarrow}
		\mu^*$ and $\mathrm{div}\, \mu_n \rightharpoonup \mathrm{div} \, \mu^*$ in $L^p(\Omega)$,
		and the continuity of $\mathcal{J}$ for  $\nu_n \to \nu $ in $\mathrm{M}^p_\Gamma(\Omega;\div)$
		that 
		\begin{equation*}
			\mathcal{J}(\mu^*)\leq \liminf_n \mathcal{J}(\mu_n) \leq \liminf_n \mathcal{J}(\nu_n) = \lim_n \mathcal{J}(\nu_n)= \mathcal{J}(\nu)
		\end{equation*}
		as $n \to \infty$. Since $\nu\in \mathbf{K}^p_{\Gamma}(\alpha ; \mathrm{div})$ was arbitrary, $\mu^*$ is a minimizer 
		for $\mathcal{J}$ and, as a result, solves $\mathbb{P}(\alpha)$.
	\end{proof}

\section{Conclusion}

We have developed several set convergence results associated to spaces of measures that include measures with divergences (functional or measure-valued) and directionally homogeneous boundary conditions. Further, we have provided the first stability results for optimization problems including such spaces.

	\bibliographystyle{abbrv}
	\bibliography{Chisholm-Rautenberg_Measures}

\end{document}